\documentclass[11pt]{article}
\usepackage{IEEEtrantools}
\usepackage{latexsym}
\usepackage{amsfonts}
\usepackage{amssymb}
\usepackage{amsmath,amsfonts,amsthm,bbm,mathtools,bm}
\usepackage{mathrsfs}
\usepackage{enumerate}
\usepackage{dsfont}
\usepackage{makecell}
\usepackage{tikz-cd}

\usepackage{multirow}
\usepackage{tikz}
\usetikzlibrary{arrows,decorations.markings}
\usepackage{bookmark}
\usepackage{hyperref}
\hypersetup{
	colorlinks   = true,
	citecolor    = blue
}

\newcommand{\bbbt}{\mathbb{T}}

\newcommand{\etale}{étalé}
\newcommand{\be}{\begin{equation}}
	\newcommand{\ee}{\end{equation}}
\newcommand{\bea}{\begin{eqnarray}}
	\newcommand{\eea}{\end{eqnarray}}
\newcommand{\bean}{\begin{eqnarray*}}
	\newcommand{\eean}{\end{eqnarray*}}
\newcommand{\brray}{\begin{array}}
	\newcommand{\erray}{\end{array}}
\newcommand{\biearray}{\begin{IEEEarray}{rCl}}
	\newcommand{\eiearray}{\end{IEEEarray}}

\newtheorem{dfn}{Definition}[section]
\newtheorem{thm}[dfn]{Theorem}
\newtheorem{lmma}[dfn]{Lemma}
\newtheorem{ppsn}[dfn]{Proposition}
\newtheorem{crlre}[dfn]{Corollary}
\newtheorem{xmpl}[dfn]{Example}
\newtheorem{rmrk}[dfn]{Remark}

\newcommand{\bdfn}{\begin{dfn}\rm}
	\newcommand{\bthm}{\begin{thm}}
		\newcommand{\blmma}{\begin{lmma}}
			\newcommand{\bppsn}{\begin{ppsn}}
				\newcommand{\bcrlre}{\begin{crlre}}
					\newcommand{\bxmpl}{\begin{xmpl}}
						\newcommand{\brmrk}{\begin{rmrk}\rm}
							
							\newcommand{\edfn}{\end{dfn}}
						\newcommand{\ethm}{\end{thm}}
					\newcommand{\elmma}{\end{lmma}}
				\newcommand{\eppsn}{\end{ppsn}}
			\newcommand{\ecrlre}{\end{crlre}}
		\newcommand{\exmpl}{\end{xmpl}}
	\newcommand{\ermrk}{\end{rmrk}}

\newcommand{\bbc}{\mathbb{C}}
\newcommand{\bbz}{\mathbb{Z}}
\newcommand{\bbn}{\mathbb{N}}

\newcommand{\bbone}{\mathbbm{1}}

\newcommand{\scrh}{\mathscr{H}}

\newcommand{\clh}{\mathcal{H}}

\newcommand{\clk}{\mathcal{K}}
\newcommand{\cll}{\mathcal{L}}

\newcommand{\clg}{\mathcal{G}}

\begin{document}

	
	\author{\sc{Shreema Subhash Bhatt, Vinay Deshpande and Bipul Saurabh}}
	\title{$C(SO_q(4)/SO_q(2))$ as a Groupoid $C^*$-algebra}
	\maketitle
	
	\begin{abstract}
		In this paper, we prove that $C(SO_q(4)/SO_q(2))$ is isomorphic to the $C^*$-algebra of the tight groupoid $\mathcal{G}_{\mathrm{tight}}$ associated with the inverse semigroup generated by the standard generators of its crystal limit $C(SO_0(4)/SO_0(2))$. We show that all four orbits of the unit space $\mathcal{G}_{\mathrm{tight}}^{(0)}$ under the natural action of $\mathcal{G}_{\mathrm{tight}}$ are locally closed, and that the associated isotropy groups are isomorphic to $\mathbb{Z}$. Consequently, every irreducible representation of $C^*(\mathcal{G}_{\mathrm{tight}})$ is induced from an irreducible representation of $C^*(\mathbb{Z})$, which are parametrized by $\mathbb{T}$. In this way, we obtain four families of irreducible representations parametrized by $\mathbb{T}$, and we explicitly construct their equivalence with the corresponding Soibelman irreducible representations of $C(SO_q(4)/SO_q(2))$.
	\end{abstract}
	
	\bigskip

	{\bf AMS Subject Classification No.:} {\large 22}A{\large 22},  {\large 46}L{\large 05},
	{\large 46}L{\large 67}. \\

	\bigskip
	\section{Introduction}
	The construction of the $q$-deformation $G_q$ of a semisimple, simply connected compact Lie group $G$ \cite{KorSoi-1998a}, together with its quotient spaces, provides a rich source of interesting noncommutative spaces that one would naturally like to place within Connes framework of noncommutative geometry \cite{Con-1994a}. To this end, understanding the underlying $C^*$-algebras is crucial; however, many basic questions concerning their topological structure remain unanswered. For instance, although their $K$-groups agree with those of the classical analogues due to $KK$-equivalence, there is still no explicit description of the generators of these groups. Moreover, the question of $q$-invariance of these quotient spaces is settled only in a handful of cases.  One of the main obstacles is that, in the faithful Soibelman representation, the images of the generators of the underlying $C^*$-algebra involve sums of operators, which are difficult to analyze directly.
	One possible way to address this difficulty is to realize the underlying $C^*$-algebras as $C^*$-algebras of combinatorial objects such as directed graphs, groupoids, or inverse semigroups. In many cases, it is easier to study these combinatorial models first, and then use them to understand the structural and topological properties of the associated $C^*$-algebras. This approach has been used successfully in several situations.
	In \cite{HonSzy-2002aa}, Hong and Szymański showed that various quantum spaces, including quantum real projective spaces and quantum complex projective spaces, can be realized as graph $C^*$-algebras. Sheu employed groupoid techniques in \cite{She-1997a} to analyze the structure of $C(SU_q(n))$, and in \cite{She-1997b} proved that the quotient spaces 
	$C(S_q^{2n-1}) := C(SU_q(n)/SU_q(n-1))$ 
	are groupoid $C^*$-algebras.
	An alternative proof of Sheu's work was given by Sundar in \cite{Sun-2013a}, where he used the theory of inverse semigroups developed in \cite{Exe-2008a} to construct the groupoid  associated with $C(S_{q}^{2n-1})$, and proved that the groupoid is isomorphic to Sheu's groupoid obtained in \cite{She-1997b}.

	Let us denote by $B_n^{n-1}$, $C_n^{n-1}$, and $D_n^{n-1}$ the families of quantum homogeneous quotient spaces
	$
	SO_q(2n+1)/SO_q(2n-1),
	SP_q(n)/SP_q(n-1), 
	SO_q(2n)/SO_q(2n-2),
	$
	respectively. 
	A careful analysis of the diagram of the Soibelman faithful representation associated with the longest Weyl word shows a common structural feature in all three families. In each case, the diagram contains a fundamental building block at its core, and the pattern by which the diagram expands on either side of this block is identical across the families $B_n^{n-1}$, $C_n^{n-1}$, and $D_n^{n-1}$. 
	Thus, in order to understand these quotient spaces, it is essential to analyze the underlying basic building blocks. These blocks arise from elementary representations corresponding either to the simple reflection $s_n$ alone or to the pair $s_n$ and $s_{n-1}$, where $n$ denotes the rank. Consequently, one must first examine the cases of small rank, where only these fundamental components appear.
	For type $C$, the $C^*$-algebra $C(H_q^1)$ associated with the basic block is isomorphic to $C(SU_q(2))$. In \cite{Sau-2017b}, it was proved that 
	\[
	C(SP_q(n)/SP_q(n-1)) \cong \Sigma^{2(n-1)} C(SU_q(2)).
	\]
	For type $B$, the $C^*$-algebra $C(SO_q(3))$ was studied in \cite{Lan-1998aa}, and in \cite{BhuBisSau-2024aaa} it was shown that \[ C(SO_q(3)) \cong \Sigma_2^2 C(\mathbb{T}), \] where $\Sigma_2^2$ denotes the quantum $2$-torsioned double suspension. More recently, it has been proven in \cite{Sau-2026a} that \[ C\big(SO_q(2n+1)/SO_q(2n-1)\big) \cong \Sigma^{2(n-1)} \, \Sigma_2^2 \, \Sigma^{2(n-1)} C(\mathbb{T}). \] 
	This naturally raises the following question: does an analogous result hold for the type $D$ family $D_n^{n-1}$? In the present article, we initiate the study of this problem by analyzing the basic building block of the type $D$ family, namely, $C\big(SO_q(4)/SO_q(2)\big)$. The higher-rank cases will be addressed in a subsequent work.
	
	In this paragraph, we briefly summarize the main results obtained. Since $C(SO_{q}(4)/SO_{q}(2))$ and $C(SO_{0}(4)/SO_{0}(2))$ are isomorphic (see \cite{BhuBisSau-2024aaa}), we work with the standard generators of $C(SO_{0}(4)/SO_{0}(2))$, which are partial isometries. The collection of all words formed by these partial isometries constitutes an inverse semigroup. Using Exel's approach, we explicitly construct a groupoid $\mathcal{G}_{\mathrm{tight}}$ generated by this inverse semigroup.
	We show that the groupoid $\mathcal{G}_{\mathrm{tight}}$ is locally compact, Hausdorff, \'etale, and amenable, and that it admits a left Haar system consisting of counting measures. Moreover, its unit space is $\overline{\mathbb{N}}^{2}$, where $\overline{\mathbb{N}}$ denotes the one-point compactification of $\mathbb{N}$. The action of $\mathcal{G}_{\mathrm{tight}}$ on its unit space decomposes it into four locally closed orbits: $\{(\infty,\infty)\}$, $\mathbb{N}\times\{\infty\}$, $\{\infty\}\times\mathbb{N}$, and $\mathbb{N}\times\mathbb{N}$. The isotropy group at each unit turns out to be isomorphic to $\mathbb{Z}$.
	We establish a faithful isomorphism between $C^{*}(\mathcal{G}_{\mathrm{tight}})$ and $C(SO_{q}(4)/SO_{q}(2))$. Since the orbits of the unit space are locally closed, the irreducible representations of $C^*(\mathcal{G}_{\mathrm{tight}})$ can be induced from its isotropy groups $\mathbb{Z}$. The representations of $C^*(\mathbb{Z})$ are parametrized by $\mathbb{T}$; hence, we obtain four families of irreducible representations of $C^{*}(\mathcal{G}_{\mathrm{tight}})$, parametrized by $\mathbb{T}$. Finally, we provide an explicit equivalence between these irreducible representations and those obtained in \cite{BhuBisSau-2024aaa}.
	
	Here is how this article is organized.  In Section \ref{preliminaries}, we recall the preliminaries on groupoids and inverse semigroups, and the process of constructing a tight groupoid from an inverse semigroup. We also explain groupoid $C^*$-algebras and the representations induced from their isotropy groups. In Section \ref{sect_gpoid}, we construct the groupoid $\mathcal{G}_{\mathrm{tight}}$ associated with the inverse semigroup $T$ of $C(SO_{0}(4)/SO_{0}(2))$.  Section \ref{sect_gpoid_prop} deals with the properties of the groupoid $\mathcal{G}_{\mathrm{tight}}$. The $C^*$-algebra $C^*(\clg_{\text {tight}})$ is shown to be isomorphic to $C(SO_{0}(4)/SO_{0}(2))$ in Section \ref{sect_isomism}. In Section \ref{sect_rep}, we give an explicit description of the irreducible representations of $C^*(\clg_{\text {tight}})$ induced from its isotropy groups.
	
	A word about notations: $\mathbb{N}$ is the set of all non-negative integers, and $\mathbb{Z}$ is the set of all integers. Let $\overline{\bbn}$ be the one point compactification of $\bbn$. For a discrete space $X$, let $\left\{e_x:x\in X\right\}$ denote the standard orthonormal basis for $\ell^2(X)$. The symbols $S$ and $\bm{t}$ denote the left shift on $\ell^{2}(\mathbb{N})$ and $\ell^{2}(\mathbb{Z})$, respectively. The symbol $p_{ij}$ denotes the map from $e_{i}$ to $e_{j}$, and the projection $p_{ii}$ will be denoted by $p_{i}$, while the projection $p_{00}$ will be denoted by $p$. The space of bounded operators on the Hilbert space $\scrh$ will be denoted by $\cll(\scrh)$. The symbol $1_{A}$ stands for the characteristic function on $A$. The symbol $X\overset{\phi}\cong Y$ stands for a bijection $\phi$ from the space $X$ to the space $Y$. 	
	
	\section{Preliminaries}\label{preliminaries}
	In this section, we recall the definitions of inverse semigroups, groupoids, and groupoid \( C^{*} \)-algebras from \cite{Pat-1998a, Ren-1980a}. We recall the construction of a groupoid from an inverse semigroup as described in \cite{Sun-2013a}. Further, we recollect from \cite{Wil-2019a}, the construction of the induced representations of the $C^{*}$-algebra of a groupoid from the representations of $C^{*}$-algebra of its isotropy groups.	
	\subsection{Groupoid $C^*$-algebras}	
	\begin{dfn}
		A \emph{groupoid} is a non-empty set \(\mathcal{G}\) equipped with
		\begin{itemize}
			\item a subset \(\mathcal{G}^{2}\subseteq \mathcal{G}\times \mathcal{G}\) of composable pairs,
			\item a multiplication map \(\mathcal{G}^{2}\to \mathcal{G};\,(a,b)\mapsto ab\),
			\item an inverse map \(\mathcal{G}\to \mathcal{G};\,a\mapsto a^{-1}\),
		\end{itemize}
		such that:
		\begin{enumerate}[(i)]
			\item \textbf{Associativity:} If \((a,b),(b,c)\in\mathcal{G}^{2}\), then \((ab,c)$ and $(a,bc)\in\mathcal{G}^{2}\), and \((ab)c=a(bc)\).
			\item \textbf{Involution:} For any \(a\in\mathcal{G}\),\((a^{-1})^{-1}=a\) and  \((a^{-1},a),(a,a^{-1})\in\mathcal{G}^{2}\).
			\item \textbf{Identities:} If \((a,b)\in\mathcal{G}^{2}\), then \(abb^{-1}=a\) and \(a^{-1}ab=b\).
		\end{enumerate}
		A groupoid \(\mathcal{G}\) is a \emph{topological groupoid} if \(\mathcal{G}\) is a topological space and the multiplication and inverse maps are continuous.
	\end{dfn}
	
	\begin{dfn}
		Given a groupoid \(\mathcal{G}\), the \emph{unit space} of \(\mathcal{G}\) is defined as
		\[
		\mathcal{G}^{0}=\{aa^{-1}:a\in\mathcal{G}\}.
		\]
	\end{dfn}
	
	The \emph{range} and \emph{source} maps are defined by
	\[
	r:\mathcal{G}\to\mathcal{G}^{0},\quad a\mapsto aa^{-1}
	\qquad\mbox{and}\quad
	s:\mathcal{G}\to\mathcal{G}^{0},\quad a\mapsto a^{-1}a.
	\]
	
	\begin{dfn}
		A locally compact Hausdorff groupoid \(\mathcal{G}\) is called \emph{r-discrete} if \(\mathcal{G}^{0}\) is open in \(\mathcal{G}\). A locally compact groupoid is called \emph{étale} if the range map is a local homeomorphism.
	\end{dfn}

	\begin{dfn}
		Let $\mathcal{G}$ be a groupoid and let $T \subset \mathcal{G}^{0}$.  
		The \emph{reduction groupoid} $\mathcal{G}_{T}$ is defined by
		\[
		\mathcal{G}_{T} = \{\gamma \in \mathcal{G} : r(\gamma), s(\gamma) \in T\}.
		\]
		The unit space of $\mathcal{G}_{T}$ is $T$.
	\end{dfn}

	\begin{dfn}
		For $u \in \mathcal{G}^{0}$, define:
		\begin{enumerate}[(i)]
			\item $\mathcal{G}^{u} = r^{-1}(u) = \{\gamma \in \mathcal{G} : r(\gamma) = u\}$,
			\item $\mathcal{G}_{u} = s^{-1}(u) = \{\gamma \in \mathcal{G} : s(\gamma) = u\}$,
			\item $\mathcal{G}_{u}^{u} = \mathcal{G}^{u} \cap \mathcal{G}_{u}$, called the \emph{isotropy group at $u$}.
		\end{enumerate}
	\end{dfn}

	\begin{dfn}
		A \emph{Haar system} on a locally compact Hausdorff groupoid $\mathcal{G}$
		is a family of Radon measures $\{\mu^{u}\}_{u \in \mathcal{G}^{0}}$ such that:
		\begin{enumerate}[(i)]
			\item $\operatorname{supp}(\mu^{u}) = \mathcal{G}^{u}$,
			\item for every $f \in C_{c}(\mathcal{G})$, the map  
			\[
			u \mapsto \int_{\mathcal{G}} f(\gamma)\, d\mu^{u}(\gamma)
			\]
			is continuous,
			\item for every $\eta \in \mathcal{G}$ and every $f \in C_{c}(\mathcal{G})$,
			\[
			\int_{\mathcal{G}} f(\gamma)\, d\mu^{r(\eta)}(\gamma)
			= \int_{\mathcal{G}} f(\eta\gamma)\, d\mu^{s(\eta)}(\gamma).
			\]
		\end{enumerate}
	\end{dfn}

	\noindent
	Let $\mathcal{G}$ be a locally compact Hausdorff groupoid equipped with a Haar system
	$\{\mu^{u}\}_{u\in \mathcal{G}^{0}}$.
	The *-algebra $C_{c}(\mathcal{G})$ of compactly supported continuous complex-valued functions
	on $\mathcal{G}$ is equipped with convolution and involution
	\[
	(f * g)(\gamma)
	= \int_{\mathcal{G}} f(\eta)\, g(\eta^{-1}\gamma)\, d\mu^{r(\eta)}(\eta)\,\,\mbox{and}\,\, f^{*}(\gamma) = \overline{f(\gamma^{-1})},
	\]
	respectively. For $f \in C_{c}(\mathcal{G})$, define the norms
	\[
	\|f\|_{I,r}
	= \sup_{u \in \mathcal{G}^{0}}
	\int_{\mathcal{G}} |f(\gamma)|\, d\mu^{u}(\gamma),\,
	\|f\|_{I,s}
	= \sup_{u \in \mathcal{G}^{0}}
	\int_{\mathcal{G}} |f(\gamma^{-1})|,\, d\mu^{u}(\gamma),\|f\|_{I} = \max \{\|f\|_{I,r},\|f\|_{I,s}\}.
	\]

	\begin{dfn}
		A $*$-homomorphism
		\[
		\pi : C_{c}(\mathcal{G}) \to \mathcal{B}(\scrh_{\pi})
		\]
		is called an \emph{$I$-norm bounded representation} if  
		\[
		\|\pi(f)\|_{\mathcal{B}(\scrh_{\pi})} = \|f\|_{\pi} \le \|f\|_{I}
		\quad \text{for all } f \in C_{c}(\mathcal{G}).
		\]
	\end{dfn}

	The \emph{full norm} on $C_{c}(\mathcal{G})$ is defined by
	\[
	\|f\| = \sup \{\|\pi(f)\| :
	\pi \text{ is an $I$-norm bounded representation}\}.
	\]

	\begin{dfn}
		The \emph{groupoid $C^{*}$-algebra} $C^{*}(\mathcal{G})$ is the completion
		of $C_{c}(\mathcal{G})$ with respect to the full norm $\|\cdot\|$.
	\end{dfn}

	\noindent We now define a class of $I$-norm bounded representations, denoted by
	$\operatorname{Ind}_{\delta_u}$, and use them to construct a subalgebra of
	$C^{*}(\mathcal{G})$ called the \emph{reduced groupoid $C^{*}$-algebra}
	$C_{\mathrm{red}}^{*}(\mathcal{G})$.

	\begin{dfn}\label{inddefn}
		Let $\mu$ be a Radon measure on $\mathcal{G}^{0}$.  
		Define measures $\nu$ and $\nu^{-1}$ on $\mathcal{G}$ by
		\[
		\int_{\mathcal{G}} f \, d\nu
		:= \int_{\mathcal{G}^{0}} \left( \int_{\mathcal{G}^{u}}
		f(\gamma)\, d\lambda^{u}(\gamma) \right) d\mu(u),
		\qquad
		\int_{\mathcal{G}} f \, d\nu^{-1}
		:= \int_{\mathcal{G}} f(\gamma^{-1})\, d\nu(\gamma).
		\]
		
		Define a $*$-homomorphism
		$
		\operatorname{Ind}_{\mu} : C_{c}(\mathcal{G})
		\longrightarrow B\big(L^{2}(\mathcal{G},\nu^{-1})\big)
		$
		as 
		\[
		\operatorname{Ind}_{\mu}(f)(\xi)(x)
		:= \int_{\mathcal{G}} f(y)\,\xi(y^{-1}x)\, d\lambda^{r(x)}(y),
		\qquad
		f \in C_{c}(\mathcal{G}),\ \xi \in L^{2}(\mathcal{G},\nu^{-1}).
		\]
		Then $\operatorname{Ind}_{\mu}$ is an $I$-norm bounded representation.
		
		\noindent For the point mass measure $\delta_{u}$ on $\mathcal{G}^{0}$, we obtain a
		representation
		$
		\operatorname{Ind}_{\delta_{u}}
		: C_{c}(\mathcal{G}) \longrightarrow B\big(L^{2}(\mathcal{G}_{u},\nu^{-1})\big)
		$
		given by
		\[
		\operatorname{Ind}_{\delta_{u}}(f)(\xi)(x)
		= \int_{\mathcal{G}} f(y)\,\xi(y^{-1}x)\, d\lambda^{r(x)}(y),
		\qquad
		\xi \in L^{2}(\mathcal{G}_{u},\nu^{-1}).
		\]
		
		\noindent If $\mathcal{G}$ is étale, the integral reduces to the following sum,
		\begin{equation}\label{discrete ind}
			\operatorname{Ind}_{\delta_{u}}(f)(\xi)(x)
			= \sum_{y \in \mathcal{G}^{r(x)}} f(y)\,\xi(y^{-1}x),
			\qquad
			\xi \in \ell^{2}(\mathcal{G}_{u}).
		\end{equation}
		
		\noindent The \emph{reduced norm} on $C_{c}(\mathcal{G})$ is defined by
		\[
		\|f\|_{\mathrm{r}}
		= \sup\big\{\,\|\operatorname{Ind}_{\delta_{u}}(f)\| :
		u \in \mathcal{G}^{0}\,\big\}.
		\]
		The completion of $C_{c}(\mathcal{G})$ with respect to the reduced norm is called the
		\emph{reduced groupoid $C^{*}$-algebra} and is denoted by
		$C_{\mathrm{red}}^{*}(\mathcal{G})$.
	\end{dfn}

	\begin{dfn}
		A locally compact groupoid $\mathcal{G}$ with Haar system
		$\{\lambda^{u}\}_{u \in \mathcal{G}^{0}}$ is called \emph{topologically amenable}
		if there exists a net $\{f_{i}\} \subset C_{c}(\mathcal{G})$ such that:
		\begin{enumerate}[(i)]
			\item the functions
			\[
			u \longmapsto \int_{\mathcal{G}} |f_{i}(\gamma)|^{2}\, d\lambda^{u}(\gamma)
			= (f_{i} * f_{i}^{*})(u)
			\]
			are uniformly bounded on $\mathcal{G}^{0}$,
			
			\item $f_{i} * f_{i}^{*} \longrightarrow 1$ uniformly on compact subsets of
			$\mathcal{G}$.
		\end{enumerate}
	\end{dfn}
	\noindent We will need the following results from \cite{Wil-2019a}.

	\begin{ppsn}\label{faithful}
		Let $U$ be the largest open invariant subset of $\mathcal{G}^{0}$ such that
		$\mu(U)=0$.  
		If $U = \emptyset$, then $\operatorname{Ind}_{\mu}$ is faithful on
		$C_{\mathrm{red}}^{*}(\mathcal{G})$.
	\end{ppsn}
	\begin{thm} 
		If a groupoid $\mathcal{G}$ is topologically amenable, then the reduced norm on
		$C_{c}(\mathcal{G})$ coincides with the universal $C^{*}$-norm.
	\end{thm}
	\subsection{Groupoid associated to an Inverse semigroup}
	In this subsection, we recall from \cite{Sun-2013a} the process of associating a groupoid to an inverse semigroup.

	\begin{dfn}
		A semigroup $T$ is an \emph{inverse semigroup} if every $a \in T$ has a unique associated element $a^{*} \in T$ satisfying $aa^{*}a = a$ and $a^{*}aa^{*} = a^{*}$. If there exists an element $\bm{0} \in T$ such that $\bm{0}s = s\bm{0}=\bm{0}$ for every $s \in T$, then $T$ is said to have a \emph{zero element}. If there exists an element $\bm{1} \in T$ such that $\bm{1}s = s\bm{1} = s$ for every $s \in T$, then $T$ is said to have a \emph{multiplicative identity}.
	\end{dfn}

	\noindent Observe that the map $\ast : T \to T$, defined by $a \mapsto a^{*}$, is an involution on $T$; that is, $(xy)^{*} = y^{*}x^{*}$ and $x^{**} = x$ for all $x, y \in T$.

	\begin{dfn}
		An element $e \in T$ is an \emph{idempotent} if $e^{2} = e$. An idempotent $e$ is called a \emph{projection} if $e^{2} = e^{*} = e$.
	\end{dfn}

	\noindent Let $E$ denote the set of all projections in $T$. Then $E$ is a commutative semigroup, and each idempotent is a projection in an inverse semigroup. Any inverse semigroup $S$ has a natural partial order defined by $x \le y$ if $x = yx^{*}x$. The partial order on $E$ inherited from $T$ can be written as $x \le y$ if $x = ye$ for some $e \in E$.
	
	\begin{dfn}
		A nonzero semigroup homomorphism $x : E \to \{0,1\}$ satisfying $x(0) = 0$ is called a \emph{character} on $E$.
	\end{dfn}

	\noindent Let $\widehat{E}_{0}$ be the set of all characters on $E$. Equipped with the subspace topology inherited from the product topology on $\{0,1\}^{E}$, the space $\widehat{E}_{0}$ becomes locally compact and Hausdorff.

	\begin{dfn}\label{Filter}
		Let $T$ be an inverse semigroup and $E$ its set of projections. A \emph{filter} in $E$ is a subset $A \subseteq E$ such that:
		\begin{enumerate}[(i)]
			\item $0 \notin A$,
			\item if $e \in A$ and $f \in E$ with $f \ge e$, then $f \in A$,
			\item if $e, f \in A$, then $ef \in A$.
		\end{enumerate}
		A maximal filter is called an \emph{ultrafilter}.
	\end{dfn}
	
	\begin{rmrk}
		Throughout this article, we consider only those inverse semigroups $T$ which contain a $\bm{1}$ element. Then by property (ii) of Definition \ref{Filter}, $\bm{1}$ belongs to every filter.
	\end{rmrk}

	\noindent If $x \in \widehat{E}_{0}$, then $A_{x} = \{e \in E : x(e) = 1\} = x^{-1}(\{1\})$ is a filter. Conversely, if $A$ is a filter, its characteristic function $1_{A}$ is a character. Thus, the correspondence $A \mapsto 1_{A}$ defines a bijection between the set of filters and the set of characters. Under this bijection, \emph{maximal characters} correspond exactly to ultrafilters.

	\noindent Define the set of maximal characters as 
	\[
	\widehat{E}_{\infty} = \{x \in \widehat{E}_{0} : A_{x} \text{ is an ultrafilter}\}.
	\]
	Let $\widehat{E}_{\mathrm{tight}}$ denote the closure of $\widehat{E}_{\infty}$ in $\widehat{E}_{0}$. Elements of $\widehat{E}_{\mathrm{tight}}$ are called \emph{tight characters} of $E$.

	\noindent For $x \in \widehat{E}_{0}$ and $s \in T$, define the map 
	\[
	x\cdot s : E \to \{0,1\}, \qquad (x\cdot s)(e) = x(ses^{*}).
	\]
	Then $x\cdot s(0) = 0$, and $x.s$ is a semigroup homomorphism. Moreover, $x\cdot s$ is a character if and only if $x(ss^{*}) = 1$.

	\noindent For $s \in T$, the domain and range sets are defined as 
	\[
	D_{s} = \{x \in \widehat{E}_{0} : x(ss^{*}) = 1\}, \qquad R_{s} = \{x \in \widehat{E}_{0} : x(s^{*}s) = 1\}.
	\]
	If $e \in E$, then $D_{e} = \{x \in \widehat{E}_{0} : x(e) = 1\}$. Both $D_s$ and $R_s$ are compact and open in $\widehat{E}_{0}$. The map $\alpha_s : D_{s} \to R_{s}$ defined by $\alpha_s(x) = x\cdot s$ is a homeomorphism with inverse $\alpha_{s^{*}}$.

	\noindent Consider the \emph{transformation groupoid}
	\[
	\Sigma = \{(x,s) : x \in D_s,\ s \in T\},
	\]
	where the elements of $\Sigma^{2}$ are precisely the pairs $(x,s)(x.s,t)$, and the groupoid operations are
	\[
	(x,s)(x.s,t) = (x,st), \qquad (x,s)^{-1} = (x\cdot s, s^{*}).
	\]
	Define an equivalence relation $\sim$ on $\Sigma$ by
	\[
	(x,s) \sim (x,t) \iff x(e) = x(ee^{*}) = 1 \text{ and } es = et \text{ for some } e \in E.
	\]
	The quotient $\mathcal{G}' = \Sigma / \sim$ becomes a groupoid with operations
	\[
	[(x,s)][(x\cdot s,t)] = [(x,st)],
	\qquad [(x,s)]^{-1} = [(x\cdot s, s^{*})],
	\]
	whenever $([(x,s)], [(x\cdot s,t)]) \in (\mathcal{G}')^{2}$.

	\noindent For $s \in T$ and open $U \subseteq D_{s}$, define
	$\Theta(U,s) := \{[(x,s)] : x \in U\},$
	and write $\Theta_{s}$ for $\Theta(D_{s},s)$.

	\begin{ppsn}\label{gpoid basis}
		The collection $\{\Theta(U,s) : s \in T,\ U \subseteq D_{s} \text{ open}\}$ forms a basis for a topology on $\mathcal{G}'$. With this topology, $\mathcal{G}'$ becomes a topological groupoid with unit space $\mathcal{G}'^{0} \cong \widehat{E}_{0}$ under the identification $[(x,1)] \mapsto x$. The following statements hold.
		\begin{enumerate}[(i)]
			\item For $s,t \in T$, one has $\Theta_{st} = \Theta_{s} \Theta_{t}$.
			\item For $s \in T$, $\Theta_{s}^{-1} = \Theta_{s^{*}}$.
			\item The algebra $C^{*}(\mathcal{G}')$ is generated by the set $\{1_{\Theta_{s}} : s \in T\}$.
		\end{enumerate}
	\end{ppsn}

	\noindent As a subspace of $\mathcal{G}'$, one can identify $\widehat{E_{0}}$ with the open set $\Theta(1,\widehat{E}_{0})$ under the map $[(x,1)] \leftrightarrow x$. Hence $\widehat{E}_{0}$ is open in $\mathcal{G}'$ and so $\mathcal{G}'$ is $r$-discrete (see \cite[Page 11]{Wil-2019a}).	\begin{dfn}[Orbit Space of $\clg^0$] 
		$\clg$ acts on $\clg^0\cong\widehat{E_{0}}$ as $$[x,t]\cdot s([x,t])\coloneq r([x,t]),\quad [x,t]\in\clg,$$ i.e. $[x,t]\cdot t=x\cdot t$. Thus, the orbit of $x$ is $$G(x)=\{x\cdot t:t\in T\}.$$ We denote the orbit space of $\clg^0$ by $\clg\backslash\clg^0$. We give $\clg\backslash\clg^0$ the quotient topology.
	\end{dfn}	
	The following is a useful criterion for the amenability of a groupoid.
	\begin{thm}[\cite{Wil-2019a}]\label{crit_amenable}
		Let $\clg$ be a second countable locally compact groupoid with a Haar system such that the orbit space $\clg\backslash\clg^0$ is	 $T_0$. Then $\clg$ is amenable if and only if every isotropy group $\clg_u^u$ is amenable.
	\end{thm}
	\begin{dfn}
		The groupoid $\mathcal{G}_{\mathrm{tight}}$ is defined to be the reduction of $\mathcal{G}'$ to the subset $\widehat{E}_{\mathrm{tight}}$.
	\end{dfn}

	\noindent Since $\widehat{E}_{0}$ is open in $\mathcal{G}'$, the intersection $\widehat{E}_{\mathrm{tight}} = \widehat{E}_{0} \cap \mathcal{G}_{\mathrm{tight}}$ is open in $\mathcal{G}_{\mathrm{tight}}$. Therefore, $\mathcal{G}_{\mathrm{tight}}$ is also $r$-discrete (see \cite[Page 11]{Wil-2019a}).
	\subsection{Induced Representations}
	\noindent Given a subgroupoid $\clk$ of a groupoid $\mathcal{G}$, we construct the induced representations of $C^{*}(\mathcal{G})$ from the representations of $C^{*}(\clk)$. Let $\mathcal{G}$ be a locally compact, second countable Hausdorff groupoid and let $\clk$ be a subgroupoid of $\mathcal{G}$. Given a representation $L$ of $C^{*}(\clk)$ on the Hilbert space $\scrh_{L}$, our aim is to define a Hilbert space $\scrh_{\operatorname{Ind} L}$ and to further construct a representation $\operatorname{Ind}_{\clk}^{\mathcal{G}}L$ of $C^{*}(\mathcal{G})$ on $\scrh_{\operatorname{Ind}L}$.

	\noindent Let $\mathcal{G}_{\clk^{0}}:=s^{-1}(\clk^{0})$. Define a $C_{c}(\clk)$-valued inner product on $C_{c}(\mathcal{G}_{\clk^{0}})$,
	\begin{equation}\label{pre ip}\langle \phi,\psi\rangle_{*}(k)=\int_{\mathcal{G}}\overline{\phi(y)}\psi(yk)d\lambda_{r(k)}(y),\quad \phi,\psi\in C_{c}(\mathcal{G}_{\clk^{0}})\,,\,k\in \clk.\end{equation}
	This inner product makes $C_c(\clg_{\clk^0})$ into a pre-Hilbert $C^*(\clk)$-module. Also, $\scrh_L$, being a Hilbert space, is a Hilbert $\bbc$-module, and it has a $C^*(\clk)$ action via the representation $L$.
	$$\begin{tikzcd}[row sep=3em, column sep=3em]
		C_c(\clg) & \scrh_L \\
		C^*(\clk) \arrow[u, "{\substack{\text{pre-Hilbert}\\\text{module}}}"] \arrow[ur, bend left=20, "L"] 
		& \bbc \arrow[u, "{\substack{\text{Hilbert}\\\text{module}}}"']
	\end{tikzcd}$$
	On the algebraic tensor product $C_c(\clg_{\clk^0})\odot \scrh_L$, consider the inner product 
	\begin{equation}\label{ip ind}
		\langle \phi\otimes h_{1},\psi\otimes h_{2}\rangle=\langle L(\langle \phi,\psi\rangle_{*})h_{1},h_{2}\rangle,\quad \phi,\psi\in C_{c}(\mathcal{G}_{\clk^{0}})\,,\,h_{1},h_{2}\in \scrh_{L}.
	\end{equation} 
	Upon completing $C_c(\clg_{\clk^0})$ with respect to this inner product, we get the internal tensor product $C_c(\clg_{\clk^0})\otimes_{C^*(\clk)}\scrh_L$. It is a Hilbert $\bbc$-module, i.e. a Hilbert space, which we denote by $\scrh_{\operatorname{Ind}L}$. Finally, the representation $\operatorname{Ind}_{\clk}^{\mathcal{G}}L$ induced on $C^{*}(\mathcal{G})$ by the representation $L$ on $C^{*}(\clk)$ is given by 
	\begin{equation}\label{ind rep}
		\begin{gathered}
			\operatorname{Ind}_{\clk}^{\mathcal{G}}L:C^*(\clg)\rightarrow \mathcal{L}(\scrh_{\operatorname{Ind}L}), \\\operatorname{Ind}_{\clk}^{\mathcal{G}}L(a)(\phi\otimes h)=a\ast\phi\otimes h,\quad a\in C_{c}(\mathcal{G})\,,\,\phi\in C_{c}(\mathcal{G}_{\clk^{0}})\,,\,h\in\scrh_{L}.
		\end{gathered}
	\end{equation}\\
	We refer to \cite{Wil-2019a} for the following result on induced representations. It relates the irreducible representations of a groupoid to those of its isotropy groups.
	\begin{thm}
		Let $\clg$ be a second countable locally compact Hausdorff groupoid with a Haar system. Suppose the orbits of $\clg^0$ are locally closed. Then we have the following results.
		\begin{enumerate}[(i)]
			\item Every irreducible representation of $C^*(\clg)$ is induced from an isotropy group.
			\item Let $u,v\in\clg^0$ be from the same orbit. Then, up to equivalence, there is a one - one correspondence between the representations induced from $\clg_u^u$ and those induced from $\clg_v^v$.
			\item Let $u$ and $v$ be from distinct orbits. If $L_1$ and $L_2$ are unitary representations of $\clg^u_u$ and $\clg_v^v$ respectively, then the induced representations $\operatorname{Ind}_{\clg_u^u}^{\clg}L_1$ and $\operatorname{Ind}_{\clg_u^u}^{\clg}L_2$ are not equivalent.
			\item If $L$ is an irreducible representation of $\clg^u_u$, then $\operatorname{Ind}_{\clg_u^u}^{\clg}L$ is an irreducible representation of $C^*(\clg)$. 
		\end{enumerate}
	\end{thm}
	\section{Inverse Semigroup $T$ and the associated groupoid $\mathcal{G}_{\mathrm{tight}}$}\label{sect_gpoid}
	
	This section constructs the inverse semigroup $T$ generated by the standard generators of $C(SO_{0}(4)/SO_{0}(2))$, which are partial isometries and explicitly computes its tight groupoid  $\mathcal{G}_{\mathrm{tight}}$. \\
	
    \noindent Let $D_q \subset \cll(\ell^2(\bbn) \otimes \ell^2(\bbn) \otimes \ell^2(\bbz))$ be the $C^*$-subalgebra of $C(SO_{q}(4)/SO_{q}(2))$ generated by
	\begin{equation*}
		\begin{aligned}
			s_{1}^q &= \sqrt{1-q^{2N}}S^* \otimes \sqrt{1-q^{2N}}S^*\otimes \bm{t}, & s_{2}^q &= -q^N\otimes \sqrt{1-q^{2N}}S^*\otimes \bm{t}, \\
			s_{3}^q &= -\sqrt{1-q^{2N}}S^*\otimes q^N\otimes \bm{t}, & s_{4}^q &= -q^N\otimes q^N\otimes \bm{t}.
		\end{aligned}
	\end{equation*}
	Defining $s_{i}^{0} = \lim_{q \rightarrow 0} s_{i}^{q}$ for $1\leq i\leq 4$, a direct computation gives
	\begin{equation*}
		s_{1}^{0}= S^*\otimes S^*\otimes \bm{t}, \quad s_{2}^{0}= -p\otimes S^*\otimes \bm{t}, \quad s_{3}^{0}= -S^*\otimes p\otimes \bm{t}, \quad s_{4}^{0}= -p\otimes p\otimes \bm{t}.
	\end{equation*}
	Let $D_0$ be the $C^*$-subalgebra of $\cll(\ell^2(\bbn)\otimes \ell^2(\bbn)\otimes\ell^2(\bbz))$ generated by $\{s_{i}^{0}: 1\leq i \leq 4\}$. 
	
	\begin{ppsn}[\cite{BhuBisSau-2024aaa}]
		For all $q\in(0,1)$, we have $D_{q}=D_{0}$. 
	\end{ppsn}
	
	\noindent	Let  $\bm{n}=(n_1, n_2), \bm{m}=(m_1, m_2) \in {\mathbb N}^2$ and $r \in {\mathbb Z}$. Define
	\begin{IEEEeqnarray*}{rCll}
		B_1(r,\bm{n},\bm{m}) & = & S^*{^{n_{1}}}S^{m_{1}} \otimes S^*{^{n_{1}}}S^{m_{1}} \otimes \bm{t}^{r}\,\,;& r=m_{1}-n_{1},\\
		B_2(r,\bm{n},\bm{m}) & = & S^*{^{n_{1}}}S^{m_{1}} \otimes S^*{^{n_{2}}}pS^{m_{2}} \otimes \bm{t}^{r}\,\,;&  r=m_{1}-n_{1},n_{1}\geq n_{2},m_{1}\geq m_{2},\\
		B_3(r,\bm{n},\bm{m}) & = & S^*{^{n_{1}}}pS^{m_{1}} \otimes S^*{^{n_{2}}}S^{m_{2}} \otimes \bm{t}^{r}\,\,;& r=m_{2}-n_{2},n_{1}\leq n_{2},m_{1}\leq m_{2},\\
		B_4(r,\bm{n},\bm{m}) & = & S^*{^{n_{1}}}pS^{m_{1}} \otimes S^*{^{n_{2}}}pS^{m_{2}} \otimes \bm{t}^{r}\,\,;& r\in\mathbb{Z}.
	\end{IEEEeqnarray*}

	\noindent Each element $B_{i}(r,\bm{n},\bm{m})$ is a word from the set $\{s_{i}^{0},{s_{j}^{0}}^{*}:1\leq i<j\leq 4\}$.

	\begin{ppsn}\label{Semigroup U}
		Define
		$T=\{B_i(r,n,m) : 1 \leq i \leq 4, n, m \in {\mathbb N}^2, r \in {\mathbb Z}\}\cup\{0\}.$ Then $T$ is an inverse semigroup with respect to the operator multiplication and adjoint operation inherited from the $C^{*}$-algebra $\cll(\ell^2(\bbz)\otimes\ell^2(\bbn)\otimes \ell^2(\bbn) )$. 
	\end{ppsn}
	\begin{proof}
		Let $\bm{n} = (n_1, n_2), \bm{m} = (m_1, m_2) \in {\mathbb N}^2$ and $k_1, k_2 \in \{0, 1\}$. A routine computation shows that, 	\[
		({S^{*}}^{n_1}p^{k_1}S^{m_1})
		({S^{*}}^{n_2}p^{k_2}S^{m_2})
		=
		\bbone_{\{k_1=0,\; m_1\le n_2\}}
		{S^{*}}^{\,n_1+n_2-m_1}p^{k_2}S^{m_2}
		+
		\bbone_{\{k_2=0,\; m_1\ge n_2\}}
		{S^{*}}^{\,n_1}p^{k_1}S^{m_1-n_2+m_2}.
		\]
		
		\noindent Thus, $T$ is a semigroup. For $B_i(r, \bm{n}, \bm{m}) \in T$, we have
		\begin{eqnarray*}
			B_i(r,\bm{n}, \bm{m})^*&=&B_i(-r, \bm{m}, \bm{n}),\\
			B_i(r, \bm{n}, \bm{m})B_i(r, \bm{n}, \bm{m})^*B_i(r, \bm{n}, \bm{m})  & = & B_i(r, \bm{n}, \bm{m}),\\
			B_i(r, \bm{n}, \bm{m})^*B_i(r, \bm{n}, \bm{m})B_i(r, \bm{n}, \bm{m})^*  & = & B_i(r, \bm{n}, \bm{m}).
		\end{eqnarray*}
		Thus, $T$ is an inverse semigroup.
	\end{proof}
	\begin{lmma}\label{Projections in T}
		For $1 \leq i \leq 4$, the element $B_i(r, \bm{n}, \bm{m})$ is a projection in $T$ iff $r = 0$ and $\bm{n} = \bm{m}$.
	\end{lmma}

	\begin{proof}
		Each $B_{i}(0,\bm{n},\bm{n})$ is clearly a projection in $T$. If $B_i(r, \bm{n}, \bm{m})$ is a projection, the identity $B_i(r, \bm{n}, \bm{m}) = B_i(r, \bm{n},\bm{m})^*=B_i(-r,\bm{m},\bm{n})$ implies $r = 0$ and $\bm{n}=\bm{m}$.
	\end{proof}
	\noindent Let $P_i(\bm{n}) = B_i(0,\bm{n},\bm{n})$ for $1 \leq i \leq 4$ and $\bm{n} \in \mathbb{N}^2$. The set of idempotents in $T$ is $E = \bigcup_{i=1}^4 E_i$, where
	\begin{align*}
		E_1 &= \{P_1(\bm{n}) : n_1 = n_2\}, & E_2 &= \{P_2(\bm{n}) : n_1 \geq n_2\},\\
		E_3 &= \{P_3(\bm{n}) : n_1 \leq n_2\}, & E_4 &= \{P_4(\bm{n}) : \bm{n} \in \mathbb{N}^2\}.
	\end{align*}
	For a subsemigroup $\Lambda \subseteq E \setminus \{0\}$, define the filter 
	$$ A_{\Lambda} = \{f \in E : fe = e \text{ for some } e \in \Lambda\}. $$
	Thus $\chi_{_{A_{\Lambda}}}$ is a character on $E$. Moreover, $\Lambda_1 \subseteq \Lambda_2$ implies $A_{\Lambda_1} \subseteq A_{\Lambda_2}$.
	\begin{dfn}
		For $\bm{k}= (k_1, k_2) \in \overline{\mathbb N}^2$, define the subsemigroups $\Lambda(\bm{k})$ of $E$ as follows:
		\begin{equation}\label{LambdaK}
			\begin{aligned}	
				\Lambda(\bm{k}):=\Lambda(k_1, k_2):=&\{P_4(\bm{k})\}\,\,& & \text{if } k_1, k_2 \in {\mathbb N}, \\
				\Lambda(k_1, \infty):=&\{P_3((k_1, l_2)) : k_1 \leq l_2 \in {\mathbb N}\}\,\,& &\text{if } k_1 \in {\mathbb N}, k_2 = \infty,\\
				\Lambda(\infty, k_2):=& \{P_2((l_1, k_2)) : k_2 \leq l_1 \in {\mathbb N}\}\,\,& &\text{if } k_1 = \infty, k_2 \in {\mathbb N},\\
				\Lambda(\infty, \infty):= &\{P_1((l_1, l_2)) : l_1 = l_2 \in {\mathbb N}\} = E_1 \,\, & & \text{if } k_1 =  k_2 = \infty.
			\end{aligned}	
		\end{equation}
	\end{dfn}

	\begin{equation}\label{ALambdaK}
		\begin{aligned}	
			A_{\Lambda(k_1, k_2)} & = \{P_1(\bm{n}) : n_1 = n_2 \leq \min\{k_1, k_2\}\} \cup \{P_2(\bm{n}) : k_2 = n_2 \leq n_1 \leq k_1\},& \\
			&\cup \{P_3(\bm{n}) : k_1 = n_1 \leq n_2 \leq k_2\} \cup \Lambda(k_1, k_2),&\\
			A_{\Lambda(k_1, \infty)} & =\{P_1(\bm{n}) : n_1 = n_2 \leq k_1\} \cup \Lambda(k_1, \infty) \subset E_1 \cup E_3,& \\
			A_{\Lambda(\infty, k_2)} & =\{P_1(\bm{n}) : n_1 = n_2 \leq k_2\} \cup \Lambda(\infty, k_2) \subset E_1 \cup E_2,& \\
			A_{\Lambda(\infty, \infty)} & =\Lambda(\infty, \infty) = E_1.& 
		\end{aligned}	
	\end{equation}
	\begin{rmrk}\label{Minimum Projections}
		For $k_1, k_2 \in \mathbb{N}$, the order on $E$ implies:
		\begin{enumerate}[(i)]
			\item $P_4(\bm{k})$ is the minimum element of $A_{\Lambda(\bm{k})}$ since  $P_i(\bm{n})P_4(\bm{k}) = P_4(\bm{k})$ for all $P_i(\bm{n}) \in A_{\Lambda(\bm{k})}$.
			\item $A_{\Lambda(k_1, \infty)}$, $A_{\Lambda(\infty, k_2)}$, and $A_{\Lambda(\infty, \infty)}$ possess no minimum elements.
		\end{enumerate}
		\end{rmrk}
	\begin{ppsn}\label{Image}
		The collection of all ultra filters in $E$ is precisely the set $\{A_{\Lambda(\bm{k})}:\bm{k}\in\overline{\mathbb N}^2\}$.
	\end{ppsn}
	\begin{proof}
		We begin by demonstrating that $A_{\Lambda(\bm{k})}$ is an ultrafilter for every $\bm{k} \in \overline{\mathbb{N}}^2$. 
	  Let $\mathcal{B}$ be a filter on $E$ such that $A_{\Lambda(\bm{k})} \subseteq \mathcal{B}$. Suppose, toward a contradiction, that there exists an element $P_i(\bm{m}) \in \mathcal{B} \setminus A_{\Lambda(\bm{k})}$. We prove that the product of $P_i(\bm{m}) \in \mathcal{B}$ and some element of $A_{\Lambda(\bm{k})} \subseteq \mathcal{B}$ is zero. Thus, $0 \in \mathcal{B}$, which contradicts the definition of a filter. Consider the case $k_1 \in \mathbb{N},k_2=\infty$. Hence, $\Lambda(\bm{k}) = \{P_3(k_1, l_2) : k_1 \leq l_2 \in \mathbb{N}\}$. Let $P_i(\bm{m}) \in \mathcal{B} \setminus A_{\Lambda(\bm{k})}$. Suppose $i=1$. The condition $m_1 = m_2 > k_1$ implies that $P_3(k_1, k_1) \in A_{\Lambda(\mathbf{k})} \subseteq \mathcal{B}$. Hence, $P_1(\bm{m})P_3(k_1, k_1) = 0$, as required. The other cases follow similarly as outlined in the following. Suppose $k_1, k_2 \in \mathbb{N}$. Then $P_{i}(\mathbf{m})P_{4}(\bm{k}) = 0$ for $P_4(\bm{k})\in A_{\Lambda(\bm{k})}$. Suppose $k_1 \in \mathbb{N}$, $k_2 = \infty$. Then $P_i(\bm{m})P_3(k_1, k_1) = 0$  for $i=1,3,\, P_3(k_1, k_1)\in A_{\Lambda(\bm{k})}$; and $P_i(\bm{m})P_3(k_1, l_2) = 0$  for $i=2,4,\,l_2>m_2,\, P_3(k_1, l_2)\in A_{\Lambda(\bm{k})}$. The case $k_1 = \infty$, $k_2 \in \mathbb{N}$ is similar to the previous case. Suppose $k_1=k_2=\infty$. Then $P_1(\mathbf{m})P_i(\mathbf{n}) = 0$ for $i\in\{1,2,3\},\,m_1 = m_2 > \max\{n_1, n_2\},\,P_1(\bm{m}) \in A_{\Lambda(\infty, \infty)}$; while the case $i=4$ is not possible. Hence, we conclude that $\phi(\overline{\mathbb{N}}^2) \subseteq \widehat{E}_{\infty}$. 
	  
 	  For the converse, let $y \in \widehat{E}_{\infty}$ and $A_y = y^{-1}(\{1\})$. As $A_y$ is an ultrafilter on $E$, we check four cases:   (i) $A_y \cap E_4 \neq \emptyset$; (ii) $A_y \cap E_4 = \emptyset$, $A_y \cap E_3 \neq \emptyset$; (iii) $A_y \cap (E_4 \cup E_3) = \emptyset$, $A_y \cap E_2 \neq \emptyset$; and (iv) $A_y \cap E_4 = A_y \cap E_3 = A_y \cap E_2 = \emptyset$. In each case, we prove that $y=\phi(\bm{k})$ for some $\bm{k}\in\overline{\bbn}^2$. In case (ii), observe that $A_y \subseteq E_1 \cup E_3$. Thus, either $E_1 \subseteq A_y$ or $A_{\Lambda(k_1, \infty)} \subseteq A_y$ for some $k_1 \in \mathbb{N}$. If $E_1 \subseteq A_y$, then $A_{\Lambda(\infty, \infty)} = E_1 \subseteq A_y$; by maximality, $A_y = A_{\Lambda(\infty, \infty)}$ and $y = \phi(\infty, \infty)$. Otherwise, there exists $k_1 \in \mathbb{N}$ such that $A_{\Lambda(k_1, \infty)} \subseteq A_y$. Since $A_{\Lambda(k_1, \infty)}$ is an ultrafilter, $A_y = A_{\Lambda(k_1, \infty)}$, which gives $\phi(k_1, \infty) = 1_{A_{\Lambda(k_1, \infty)}} = 1_{A_y} = y$, as required. For cases (i), (iii) and (iv), one can verify that $y=\phi(\bm{k}),\,\phi(\infty,k_2)$ and $\phi(\infty,\infty)$ work respectively, where $k_1,k_2\in\bbn$.				\end{proof}
	\begin{ppsn}\label{Homeomorphism}
		The map $\phi:\overline{\mathbb N}^2 \longrightarrow \widehat{E}_{\infty}\,;\,\phi(\bm{k}) = 1_{A_{\Lambda(\bm{k})}}$ is a homeomorphism.
	\end{ppsn}
		\begin{proof}
		By Proposition~\ref{Image}, $\phi$ is a bijection. Since $\overline{\mathbb{N}}^2$ is compact and $\widehat{E}_{\infty}$ is Hausdorff, it suffices to show $\phi$ is continuous. Let $\{\bm{k}_{\alpha}\}$ be a net in $\overline{\mathbb{N}}^2$ with $\bm{k}_{\alpha} \to \bm{k}$. Continuity follows if $\phi(\bm{k}_{\alpha})(P_{i}(n)) \to \phi(\bm{k})(P_{i}(n))$ for any projection $P_{i}(n) \in E$. 
				\begin{enumerate}[(i)]
			\item If $\bm{k} \in \mathbb{N}^2$, the net is eventually constant, so the result is immediate.
			\item If $\bm{k} = (k_1, \infty)$, then
			there exists $\beta \in \Lambda$ such that  
			$k_{\alpha 1} = k_1$ and $k_{\alpha 2} \ge n_2$ for all $\alpha \ge \beta$. Hence  
			$P_i(\bm{n}) \in A_{\Lambda(k_1, \infty)}$ if and only if  
			$P_i(\bm{n}) \in A_{\Lambda(k_1, k_{\alpha 2})}$ for all $\alpha \ge \beta$; that is,  
			\[
			\phi(\bm{k}_{\alpha})(P_i(\bm{n})) \longrightarrow \phi(\bm{k})(P_i(\bm{n})).
			\]
			\end{enumerate}
		The cases $\bm{k} = (\infty, k_2)$ and $\bm{k} = (\infty, \infty)$ are analogous. Thus, $\phi$ is a homeomorphism.
		\end{proof}

	\begin{ppsn}
		$\widehat{E}_{\mathrm{tight}} = \overline{\widehat{E}}_{\infty}=\widehat{E}_{\infty}$.
	\end{ppsn}	
	\begin{proof}
		By Proposition~\ref{Homeomorphism}, $\widehat{E}_{\infty} = \phi(\overline{\mathbb{N}}^{2})$ is compact. Since $\widehat{E}_0$ is Hausdorff, $\widehat{E}_{\infty}$ is closed in $\widehat{E}_0$, and the claim follows.
	\end{proof}

	\noindent We now determine the transformation groupoid $\Sigma|_{\widehat{E}_{\mathrm{tight}}}$.
	\begin{lmma}\label{Size of X}
		Let  $1\leq i \leq 4$, $B_i(r,\bm{n},\bm{m}) \in T$, and $\bm{k} = (k_1, k_2) \in  \overline{\mathbb{N}}^2$. Then
		\begin{enumerate}[(i)]
			\item $(\phi(\bm{k}), B_1(r,\bm{n},\bm{m})) \in \Sigma$ iff $n_1 = n_2 \leq \min\{k_1, k_2\}$.
			\item $(\phi(\bm{k}), B_2(r,\bm{n},\bm{ m})) \in \Sigma$ iff $k_2 = n_2 \leq n_1 \leq  k_1$.
			\item $(\phi(\bm{k}), B_3(r,\bm{n},\bm{m})) \in \Sigma$ iff $k_1 = n_1 \leq n_2 \leq  k_2$.
			\item $(\phi(\bm{k}), B_4(r,\bm{n},\bm{m})) \in \Sigma$ iff $n = k$.
		\end{enumerate}
	\end{lmma}

	\begin{proof}
	Let $1\leq i \leq 4$, $B_i(r,\bm{n},\bm{m}) \in T$, and $\bm{k} = (k_1, k_2) \in \overline{\mathbb{N}}^{2}$. Then,
$$
(\phi(\bm{k}), B_i(r,\bm{n},\bm{m})) \in \Sigma\text{ iff }P_i(\bm{n}) \in A_{\Lambda(\bm{k})},\quad (\phi(\bm{k}), P_i(\bm{n})) \in \Sigma\text{ iff }P_i(\bm{n}) \in A_{\Lambda(\bm{k})}.
$$		Hence,  $(\phi(\bm{k}), B_i(r,\bm{n},\bm{m})) \in\Sigma \text{ iff } P_i(\bm{n}) \in A_{\Lambda(\bm{k})}.$ The claim now follows from equation~(\ref{ALambdaK}). 
	\end{proof} 
	\begin{lmma}\label{Lemma 1}
		Let $1 \leq i \leq 4$ and $B_i(r,\textbf{n},\textbf{m}) \in T$. Then,
		\begin{enumerate}[(i)]
			\item $D_{B_1(r,\textbf{n},\textbf{m})} = D_{P_1(\textbf{n})} = \{\phi(\textbf{k}) : \textbf{k} \in  \overline{\mathbb{N}}^2, n_1 = n_2 \leq \min\{k_1, k_2\} \}$,
			\item  $D_{B_2(r,\textbf{n},\textbf{m})} = D_{P_2(\textbf{n})} = \{\phi(k_1, n_2) : k_1 \in  \overline{\mathbb{N}}, k_1 \geq n_1\geq n_{2}\}$,
			\item  $D_{B_3(r,\textbf{n},\textbf{m})} = D_{P_3(\textbf{n})} = \{\phi(n_1, k_2) : k_2 \in  \overline{\mathbb{N}}, n_{1}\leq n_2 \leq k_2\}$,
			\item  $D_{B_4(r,\textbf{n},\textbf{m})} = D_{P_4(\textbf{n})} = \{\phi(\textbf{n})\}$.
		\end{enumerate}
	\end{lmma}
	\begin{lmma}\label{FiniteFinite} The equivalence relation $\sim$ on $\Sigma$ is explicitly given below.
		
		\begin{enumerate}[(i)]
		\item Let $\bm{k}, \bm{m} \in \mathbb{N}^2$, $r \in \mathbb{Z}$. If $(\phi(\bm{k}), B_4(r, \bm{k}, \bm{m})) \in \Sigma$, then \\$[(\phi(\bm{k}), B_4(r, \bm{k}, \bm{m}))] = \bm{A}_1$, where
	    \begin{eqnarray*}
	    	&\bm{A}_{1} := & \{(\phi(\bm{k}), B_1(r',\bm{n'},\bm{m'})) \in \Sigma : r = r', \; n_1' = n_2' \leq \min\{k_1, k_2\}, \bm{m} = \bm{k}-\bm{n'}+\bm{m'}\}\\
	    	& & \cup \{(\phi(\bm{k}), B_2(r',\bm{n'},\bm{m'})) \in \Sigma : r = r', \; k_2 = n_2' \leq n_1' \leq k_1, \bm{m} = \bm{k}-\bm{n'}+\bm{m'}\}\\
	    	& & \cup \{(\phi(\bm{k}), B_3(r',\bm{n'},\bm{m'})) \in \Sigma : r = r', \; k_1 = n_1' \leq n_2' \leq k_2, \bm{m} = \bm{k}-\bm{n'}+\bm{m'}\}\\
	    	& & \cup \{(\phi(\bm{k}), B_4(r,\bm{k},\bm {m}))\}.
	    \end{eqnarray*} 
	    \item Let $\bm{k} = (k_1, n_2),\bm{m}\in\mathbb{N}^{2}$ and let $(\phi(k_1, \infty), B_3(r,\bm{k},\bm{m})) \in T$. Then $r = m_2 - n_2 \in {\mathbb Z}$. If $(\phi(k_1, \infty), B_3(r,\bm{k},\bm{m})) \in \Sigma$, then $k_1\leq n_2$ and $[(\phi(k_1, \infty), B_3(r,\bm{k},\bm{m}))]=\bm{A}_{2}$, where  \begin{eqnarray*}
	    	&\bm{A}_{2} := & \{(\phi(k_1, \infty), B_1(r',\bm n',\bm m')) \in \Sigma : r = r', \, n_1' = n_2' \leq k_1 , \, \bm{m} = \bm{k}-\bm{n'}+\bm{m'}\}\\
	    	&  & \cup \{(\phi(k_1, \infty), B_3(r',\bm{n'}, \bm{m'})) \in \Sigma : r = r', \, k_1 = n_1' \leq n_2', \, \bm{m} = \bm{k}-\bm{n'}+\bm{m'}\}.
	    \end{eqnarray*}
	    \item Let $\bm{k} = (n_1, k_2),\bm{m}\in\mathbb{N}^{2}$ and let $(\phi(\infty, k_2), B_2(r,\bm k,\bm m)) \in T$. Then $r = m_1 - n_1 \in {\mathbb Z}$. If $(\phi(\infty, k_2), B_2(r,\bm k,\bm m)) \in \Sigma$, then $k_2\leq n_1$  and \\$[(\phi(\infty, k_2), B_2(r,\bm{k},\bm{m}))]= [(\phi(\infty, k_2), B_2(r, (n_1, k_2),\bm{ m}))]=\bm{A}_{3}$, where  \begin{eqnarray*}
	    	&\bm{A}_{3} := & \{(\phi(\infty, k_2), B_1(r',\bm{n'},\bm {m'})) \in \Sigma : r = r', n_1' \leq n_2' \leq k_2 , \bm{m} = (k_1, l_2)-\bm{n'}+\bm{m'}\}\\
	    	&  & \cup \{(\phi(\infty, k_2), B_2(r',\bm{n'}, \bm{m'})) \in \Sigma : r = r', k_2 = n_2' \leq n_1', \bm{m} = (k_1, l_2)-\bm{n'}+\bm{m'}\}.
	    \end{eqnarray*}
	  \item Let $\bm{m},\bm{n}\in\mathbb{N}^{2}$. Let $(\phi(\infty, \infty), B_1(r,\bm{n},\bm{m})) \in T$. Then $r = m_1-n_1$, $n_1 = n_2, m_1 = m_2$. If $(\phi(\infty, \infty), B_1(r,\bm{n},\bm{m})) \in \Sigma$, then $[(\phi(\infty, \infty), B_1(r,\bm{n},\bm{m}))] =\bm{A}_{4}$, where   $$\bm{A}_{4}:= \{(\phi(\infty, \infty), B_1(r',\bm{n'},\bm{m'})) \in\Sigma : r = r'\}.$$
	    \end{enumerate}
	\end{lmma}
	
	\begin{proof}
		\begin{enumerate}[(i)]
		\item Note that $\phi(\bm{k})(P_4(\bm{k})) = 1$ and $P_4(\bm{k}) B_4(r, \bm{k}, \bm{m}) = B_4(r, \bm{k}, \bm{m})$. Observe that, for any $(\phi(\bm{k}), B_i(r',\bm{n'},\bm{m'})) \in \bm{A}_{1},$  we have, $$P_4(k)B_i(r',\bm{n'},\bm{m'}) =B_4(r',\bm{k},\bm{k}-\bm{n'}+\bm{m'}) = B_4(r,\bm{k},\bm{m}) = P_4(\bm{k})B_4(r, \bm{k}, \bm{m}),$$ thus establishing $\bm{A}_1 \subseteq [(\phi(\bm{k}), B_4(r, \bm{k}, \bm{m}))]$. Conversely, if $(\phi(\bm{k}), B_i(r', \bm{n}', \bm{m}'))\sim$\\$(\phi(\bm{k}), B_4(r, \bm{k}, \bm{m}))$, 
		then $\bm{l} = \bm{k}$ and there exists $P_j(\bm{h}) \in E$ such that
		$\phi(\bm{k})(P_j(\bm{h})) = 1$ and $ P_j(\bm{h})B_i(r',\bm n',\bm m') = P_j(\bm{h})B_4(r,\bm{k},\bm{m}).$ 
		Since $\phi(\bm{k})(P_j(\bm{h})) = 1$, we must have  $P_j(\bm{h}) \in A_{\Lambda(\bm{k})}$. Since $P_4(\bm{k})$ is the minimum element in $A_{\Lambda(\bm{k})}$, $P_j(\bm{h})P_4(\bm{k}) = P_4(\bm{k})$. Thus, replacing $P_j(\bm{h})$ by a larger projection $P_4(\bm{k})$ in the above identity,
		\begin{eqnarray}\label{B}
			\phi(\bm{k})(P_4(\bm{k})) = 1 \quad \textrm{ and } \quad P_4(\bm{k})B_i(r',\bm n',\bm m') = P_4(\bm{k})B_4(r,\bm{k},\bm{m}) = B_4(r,\bm{k},\bm{m}).
		\end{eqnarray}
		Observe that $P_4(\bm{k})B_i(r',\bm{n'},\bm{m'}) \neq 0$ since $B_4(r,\bm{k},\bm{m}) \neq 0$. But
		\begin{eqnarray*}
			P_4(\bm{k})B_i(r',\bm{n'},\bm{m'}) \neq 0 \Longleftrightarrow \left\{
			\begin{array}{ll}
				n_1' = n_2' \leq \min\{k_1, k_2\},\,\,&\mbox{if}\, i=1, \\
				k_2 = n_2' \leq n_1' \leq k_1,\,\,&\mbox{if}\, i=2, \\
				k_1 = n_1', \leq n_2' \leq k_2,\,\,&\mbox{if}\,i=3 \\
				k = n',\,\,&\mbox{if}\,i=4.
			\end{array}
			\right.
		\end{eqnarray*}
		In all the above cases,
		\begin{eqnarray}\label{C}
			P_4(\bm{k})B_i(r',\bm{n'},\bm{m'}) = B_4(r', \bm{k}, \bm{k}-\bm{n'}+\bm{m'}).
		\end{eqnarray}
		Equations \eqref{B} and \eqref{C} imply $r' = r$ and $\bm{m} = \bm{k} - \bm{n}' + \bm{m}'$, establishing the reverse containment and the claim.
		\item The elements $(\phi(k_1, \infty), B_1(r',\bm{n}',\bm{m}'))$ and $(\phi(k_1, \infty), B_3(r',\bm{n}',\bm{m}'))\in \bm{A}_2$, are equivalent to $(\phi(k_1,\infty),B_3(r,\bm{k},\bm{m}))$ via the projections $P_3(k_1, n_2')$ and $P_3(k_1, n_2)$, respectively. Conversely, let $(\phi(l), B_i(r',\bm{n'}, \bm{m'})) \sim (\phi(k_1, \infty), B_3(r,\bm{k}, \bm{m}))$ via some $P_j(\bm{h}) \in A_{\Lambda(k_1, \infty)}$. Hence, $l = (k_1, \infty)$. By equation~\eqref{ALambdaK}, $i,j \in \{1, 3\}$. As before, assume the larger projection, i.e., $j=3$ and  $h_2 \geq \max\{n_2, n_2'\}$. For $i\in\{1,3\}$, we have
			$$
				P_3(\bm{h})B_i(r',\bm{n'},\bm{m'}) = B_3(r', \bm{h}, \bm{h}-\bm{n'}+\bm{m'}).
			$$
			Comparing both sides, $r' = r$, $\bm{h} = \bm{k}$, and $\bm{h}-\bm{k}+\bm{m} = \bm{h}-\bm{n'}+\bm{m'}$ which implies $\bm{m} = \bm{k}-\bm{n'}+\bm{m'}$.
		\item 		The proof follows analogously to the earlier case.
		\item If $(\phi(\infty, \infty), B_1(r,\bm{n'},\bm{m'})) \in \bm{A}_{4}$, then $m_1' - n_1' = r' = r = m_1 - n_1$ and thus, $\bm{m'}-\bm{n'} = \bm{m}-\bm{n}$. Letting $\bm{l} = (l_1, l_2)$, $l_1 = l_2 \geq \max\{n_1, m_1, n_1', m_1'\}$, we obtain $\phi(\infty, \infty)(P_1(\bm{l})) = 1$ and
		\begin{eqnarray*}
			P_1(\bm{l})B_1(r,\bm{n'},\bm{m'})  =  B_1(r, \bm{l}, \bm{l}-\bm{n'}+\bm{m'}) = B_1(r,\bm{l}, \bm{l}-\bm{n}+\bm{m})  =  P_1(\bm{l})B_1(r,\bm{n},\bm{m}),
		\end{eqnarray*}
		since $\bm{m'}-\bm{n'} =\bm{m}-\bm{n}$. Hence, $\{(\phi(\infty, \infty), B_1(r',\bm{n'},\bm {m'})) \in\Sigma : r = r'\} \subseteq [\phi(\infty, \infty), B_1(r,\bm{n},\bm{m})]$.
		
		\noindent  Conversely, if $(\phi(k), B_i(r',\bm{n'}, \bm{m'})) \sim (\phi(\infty, \infty), B_1(r,\bm{n}, \bm{m}))$ via some $P_j(\bm{l})$, then $i=j=1$. Since $P_1(\bm{l})P_1(\bm{l'}) = P_1(\max\{\bm{l}, \bm{l'}\})$ and $A_{\Lambda(\infty, \infty)} = \Lambda(\infty, \infty)$ has no minimum element, we may enlarge $\bm{l}$, i.e. assume $l_1 = l_2 \geq \max\{n_1, m_1, n_1', m_1'\}$. Then,
		$$ B_1(r', \bm{l}, \bm{l}-\bm{n'}+\bm{m'}) = P_1(\bm{l})B_1(r',\bm{n'},\bm{m'}) = P_1(\bm{l})B_1(r,\bm{n},\bm{m}) = B_1(r,\bm{l}, \bm{l}-\bm{n}+\bm{m}).$$
		Hence we must have $r = r'$ which gives the reverse inclusion, completing the proof.
		\end{enumerate}
		\end{proof}
	\section{Properties of $\mathcal{G}_{\mathrm{tight}}$}\label{sect_gpoid_prop}
	In this section, we investigate the properties of the groupoid $\mathcal{G}_{\mathrm{tight}}$ obtained in the previous section. We compute the isotropy group of $\mathcal{G}_{\mathrm{tight}}$ and prove that $\mathcal{G}_{\mathrm{tight}}$ is Hausdorff.  Finally, we discuss that $\mathcal{G}_{\mathrm{tight}}$ has a left Haar system consisting of counting measures.

	\begin{lmma}\label{Unitspace}
		The unit space $\mathcal{G}_{\mathrm{tight}}^{0}$ of $\mathcal{G}_{\mathrm{tight}}$ is
		$$\mathcal{G}_{\mathrm{tight}}^{0} \cong \{\phi(k_{1},k_{2}), \, \phi(k_{1},\infty), \, \phi(\infty,k_{2}), \, \phi(\infty,\infty) : k_1, k_2 \in {\mathbb N}\} \cong \widehat{E}_{\infty}.  $$
	\end{lmma}

	\begin{proof}
		Let $\bm{k} = (k_1, k_2) \in \overline{\mathbb N}^2$. Then, by Lemma \ref{FiniteFinite}, as well as the fact that\\ $B_{i}(r,\bm{n},\bm{m})B_{i}(r,\bm{n},\bm{m})^* =P_{i}(\bm{n})$ we have,   
		$
			[(\phi(\bm{k}),B_{i}(r,\bm{n},\bm{m}))][(\phi(\bm{k}),B_{i}(r,\bm{n}, \bm{m}))]^{-1}=[(\phi(\bm{k}),P_{i}(0))] \simeq \phi(\bm{k})$.
			This completes the proof.
	\end{proof}
	\begin{lmma} \label{fibres 1}Let $\bm{k}=(k_1,k_2)\in\bbn^2$.
		\begin{enumerate}[(a)]\item The fibres $\clg_{\mathrm{tight}}^u$, for $u\in\clg_{\mathrm{tight}}^0$, are given by the following.
			\begin{enumerate}[(i)]
				\item $\begin{aligned}[t]\clg_{\mathrm{tight}}^{\phi(\infty,\infty)}&=\left\{[\phi(\infty,\infty),B_{1}(r,\bm{0},(r,r))]:r\in\bbz_+\right\}\\&\phantom{{}={}}\cup \{[\phi(\infty,\infty),B_{1}(r,(-r,-r),\bm{0})]:r\in\bbz_-\}\\&\cong\bbz.\end{aligned}$
				\item $\begin{aligned}[t]\clg_{\mathrm{tight}}^{\phi(\infty,k_2)}&=\{[\phi(\infty,k_{2}),B_{2}(r,(k_{2}+m,k_{2}),(r+k_{2}+m,m))]:m\in\bbn,r\in\bbz_+\}\\ &\phantom{{}={}}\cup\{[\phi(\infty,k_{2}),B_{2}(r,(k_{2}+m-r,k_{2}),(k_{2}+m,m))]:m\in\bbn,r\in\bbz_-\}\\&\cong\bbn\times\bbz.\end{aligned}$
				\item $\begin{aligned}[t]\clg_{\mathrm{tight}}^{\phi(k_1,\infty)}&=\{[\phi(k_{1},\infty),B_{3}(r,(k_{1},k_{1}+m),(m,r+k_{1}+m))]:m\in\bbn,r\in\bbz_+\}\\&\phantom{{}={}}\cup\{[\phi(k_{1},\infty),B_{3}(r,(k_{1},k_{1}+m-r),(m,k_{1}+m))]:m\in\bbn,r\in\bbz_-\}\\&\cong\bbn\times\bbz.\end{aligned}$
				\item $\begin{aligned}[t]\mathcal{G}_{\mathrm{tight}}^{\phi(\bm{k})}&=\{[(\phi(\bm{k}),B_{4}(r,(k_{1},n),(k_{2},m)))]:m=n,m\in\mathbb{N},r\in\mathbb{Z}\}\\&\cong\bbn\times\bbn\times\bbz\end{aligned}$
			\end{enumerate}
			\item The fibres $\clg_{\mathrm{tight},u}$, for $u\in\clg_{\mathrm{tight}}^0$, are given by the following.
			\begin{enumerate}[(i)]\item $\begin{aligned}[t]\clg_{\mathrm{tight},{\phi(\infty,\infty)}}&=\left\{\left[\phi(\infty,\infty),B_1(r, \bm{0}, (r,r)) \right] :r\in\bbz_+\right\}\\&\phantom{{}={}}\cup\left\{\left[\phi(\infty,\infty),B_1(r, (-r,-r),\bm{0}) \right] :r\in\bbz_-\right\}\\&\cong\bbz\end{aligned}.$
				\item $\begin{aligned}[t] \clg_{\mathrm{tight},\phi(\infty,k_2)}&=\left\{\left[\phi(\infty,n),B_2(r, (n+k_{2},n), (r+n+k_{2},k_{2}))\right] :n\in\bbn,r\in\bbz_+\right\}\\
					&\phantom{{}={}}\cup\{[\phi(\infty,n),B_2(r,(n+k_{2}-r,n)),(n+k_{2},k_{2})] :n\in\bbn,r\in\bbz_-\}\\&\cong\bbn\times\bbz.\end{aligned}$
				\item$\begin{aligned}[t] \clg_{\mathrm{tight},\phi(k_1,\infty)}&=\left\{\left[\phi(n,\infty),B_3(r, (n,n+k_{1}), (k_{1},r+n+k_{1}))\right] :n\in\bbn,r\in\bbz_+\right\}\\
				&\phantom{{}={}}\cup\{[\phi(n,\infty),B_3(r,(n,n+k_{1}-r)),(k_{1},n+k_{1})] :n\in\bbn,r\in\bbz_-\}\\&\cong\bbn\times\bbz.\end{aligned}$
				\item $\begin{aligned}[t]\mathcal{G}_{\mathrm{tight},\phi(\bm{k})}&=\{[(\phi(\bm n),B_{4}(r,\bm n,\bm{k}))]:\bm n\in\bbn\times\bbn,r\in\bbz\}\\&\cong\bbn\times\bbn\times\bbz\end{aligned}$
			\end{enumerate}
			\item The isotropy groups $\clg_{\mathrm{tight},u}^u$, for $u\in\clg_{\mathrm{tight}}^0$, are given by the following.
			\begin{enumerate}[(i)]
				\item $\begin{aligned}[t]\clg_{\mathrm{tight},\phi(\infty,\infty)}^{\phi(\infty,\infty)}&=\left\{[\phi(\infty,\infty),B_{1}(r,\bm{0},(r,r))]:r\in\bbz_+\right\}\\&\phantom{{}={}}\cup \{[\phi(\infty,\infty),B_{1}(r,(-r,-r),\bm{0})]:r\in\bbz_-\}\\&\cong\bbz.\end{aligned}$
				\item  $\begin{aligned}[t]\clg_{\mathrm{tight},\phi(\infty,k_2)}^{\phi(\infty,k_2)}&=\{[\phi(\infty,k_2),B_2(r, (2k_{2},k_{2}), (r+2k_{2},k_{2}))] :r\in \bbz_+\}\\&\phantom{{}={}}\cup\{[\phi(\infty,k_2),B_2(r,(2k_{2}-r,k_{2}),(2k_{2},k_{2}))] :r\in\bbz_-\}\\&\cong\bbz.\end{aligned}$
				\item $\begin{aligned}[t]\clg_{\mathrm{tight},\phi(k_1,\infty)}^{\phi(k_1,\infty)}&=\{[\phi(k_{1},\infty),B_3(r, (k_{1},2k_{1}), (k_{1},r+2k_{1}))] :r\in \bbz_+\}\\&\phantom{{}={}}\cup\{[\phi(k_{1},\infty),B_3(r,(k_{1},2k_{1}-r),(k_{1},2k_{1}))] :r\in\bbz_-\}\\&\cong\bbz.\end{aligned}$
				\item $\{[(\phi(k),B_{4}(r,\bm{k},\bm{k}))]:r\in\mathbb{Z}\}\cong\bbz$.
			\end{enumerate}
			\end{enumerate}
	\end{lmma}
	\begin{proof}
		\begin{enumerate}[(a)]\item
		Let $\bm{k} = (k_{1},k_{2}) \in \overline{\mathbb{N}}^2$ and $u = \phi(\bm{k}) \in \mathcal{G}_{\mathrm{tight}}^{0}$. Then, as in the proof of Lemma~\ref{Unitspace},
		\begin{eqnarray*}
			[(\phi(\bm{l}),B_{i}(r,\bm{n},\bm{m}))]\in \mathcal{G}_{\mathrm{tight}}^u& \Longleftrightarrow & r([(\phi(\bm{l}),B_{i}(r,\bm{n},\bm{m}))]) = u\\
			& \Longleftrightarrow & [(\phi(\bm{l}),B_{i}(r,\bm{n},\bm{m}))][(\phi(\bm{l}),B_{i}(r,\bm{n},\bm{m}))]^{-1} =\phi(\bm{k})\\
			& \Longleftrightarrow & [(\phi(\bm{l}),B_{i}(r,\bm{n},\bm{m}))][(\phi(\bm{l}),B_{i}(r,\bm{n},\bm{m}))]^{-1} = [\phi(\bm{k}), P_1(0)]\\
			& \Longleftrightarrow & [(\phi(\bm{l}),P_{i}(\bm{n}))] = [\phi(\bm{k}), P_i(\bm{n})]\\
			& \Longleftrightarrow & \bm{l}=\bm{k}.
		\end{eqnarray*}
		The claim now follows from Lemma \ref{FiniteFinite}. The proofs for part (b) and (c) are similar.
	\end{enumerate}
		\end{proof}
	\begin{lmma}\label{rmrkunitspace}
The action of $\clg_{\mathrm{tight}}$ on $\clg_{\mathrm{tight}}^0$ is described below.
		\begin{enumerate}[(i)]\label{relations}
			\item For $(\phi(\bm k), B_{4}(r,\bm k,\bm n))\in \bm{A_1}$, $\phi(\bm k)\cdot B_{4}(r,\bm k,\bm n)=\phi(\bm n)$.
			\item For $(\phi(k_{1},\infty), B_{3}(r,(k_{1},n_{2}),\bm m))\in\bm{A_2}$, $\phi(k_{1},\infty)\cdot B_{3}(r,(k_{1},n_{2}),\bm m)=\phi(m_1,\infty)$.
			\item For $(\phi(\infty,k_{2}), B_{2}(r,(n_{1},k_{2}),\bm m))\in\bm{A_3}$, $\phi(\infty,k_{2})\cdot B_{2}(r,(n_{1},k_{2}),\bm m)=\phi(\infty,m_2)$.
			\item For $(\phi(\infty,\infty), B_{1}(r,\bm{n},\bm{m}))\in \bm{A_4}$, $\phi(\infty,\infty)\cdot B_{1}(r,\bm{n},\bm{m})=\phi(\infty,\infty)$.
		\end{enumerate}
	\end{lmma}	
	\begin{proof}
The relations can be verified from Lemma \ref{FiniteFinite}.
	\end{proof}
	\begin{ppsn}
		The groupoid $\mathcal{G}_{\mathrm{tight}}$ is Hausdorff.
	\end{ppsn}

	\begin{proof}
		Let $[\phi(\bm{k}),B_{i}(r,\bm{n},\bm{m})]\neq[\phi(\bm{k'}),B_{j}(r',\bm{n'},\bm{m'})]$ in $\mathcal{G}_{\mathrm{tight}}$. By the equivalence relation, there are two cases.
		\begin{enumerate}[(i)] 
			\item Let $\phi(\bm{k})\neq\phi(\bm{k'})$ in $\widehat{E}_{\mathrm{tight}}$. Since $\widehat{E}_{\mathrm{tight}}$ is Hausdorff, there exist disjoint open sets $U$ and $V$ of  $\widehat{E}_{\mathrm{tight}}$ such that $\phi(\bm{k})\in U\subset D_{B_{i}(r,\bm{n},\bm{m})}$ and $\phi(\bm{k'})\in V\subset D_{B_{j}(r',\bm{n'},\bm{m'})}$. Thus,  $\Theta(B_{i}(r,\bm{n},\bm{m}),U)$ and $\Theta(B_{j}(r',\bm{n'},\bm{m'}),V)$ are disjoint open sets in  $\mathcal{G}_{\mathrm{tight}}$ containing $[\phi(\bm{k}),B_{i}(r,\bm{n},\bm{m})]$ and $[\phi(\bm{k'}),B_{j}(r',\bm{n'},\bm{m'})]$ resp.
			
			\item Let $\phi(\bm{k}) = \phi(\bm{k'})$. By Lemma \ref{FiniteFinite}, $[(\phi(\bm{k}),
			B_{j}(r',\bm{n'},\bm{m'}))] =[(\phi(\bm{k}), B_{i}(r'',\bm{n''},\bm{m''}))]$ for some $r''\in\mathbb{Z}$ and $\bm{n''},\bm{m''}\in\mathbb{N}^{2}$.  Hence, for every $e \in E$,$eB_{i}(r,\bm{n},\bm{m}) \neq e B_{i}(r'',\bm{n''},\bm{m''}).$ Then $$\Theta(B_{i}(r,\bm{n},\bm{m}), D_{B_{i}(r,\bm{n},\bm{m})}) \cap \Theta(B_{i}(r'',\bm{n''},\bm{m''}), D_{B_{i}(r'',\bm{n''},\bm{m''})})=\emptyset,$$ since if there is an 
			$h \in \Theta(B_{i}(r,\bm{n},\bm{m}),D_{B_{i}(r,\bm{n},\bm{m})} ) \cap \Theta(B_{i}(r'',\bm{n''},\bm{m''}),D_{B_{i}(r'',\bm{n''},\bm{m''})} )$, then there exists an element $z \in D_{B_{i}(r,\bm{n},\bm{m})} \cap D_{B_{i}(r'',\bm{n''},\bm{m''})}$ such that
			$h = [(z, B_{i}(r,\bm{n},\bm{m}))] = [(z, B_{i}(r'',\bm{n''},\bm{m''}))]$, which is a contradiction. 
		\end{enumerate}
	\end{proof}
	\begin{ppsn}\label{etale}
		$\mathcal{G}_{\mathrm{tight}}$ is an  \etale\,\, groupoid.
	\end{ppsn}
	\begin{proof}
		By \cite[page 284]{ExePar-2016aa}, the groupoid $\mathcal{G}' = \Sigma / \sim$ is a locally compact, étale groupoid and therefore admits a left Haar system consisting solely of counting measures (see \cite[page 12]{Wil-2019a}, \cite[page 11, Proposition~1.29]{Wil-2019a}). Consequently, $\mathcal{G}_{\mathrm{tight}}$ also has a left Haar system consisting only of counting measures. Hence, the range map $r : \mathcal{G}_{\mathrm{tight}} \longrightarrow \widehat{E}_{\mathrm{tight}}$
		is a local homeomorphism (see \cite[page 11, Proposition~1.29]{Wil-2019a}). This implies that $\widehat{E}_{\mathrm{tight}}$ is an étale groupoid (see \cite[page 12]{Wil-2019a}).
	\end{proof} 
	\begin{ppsn}\label{g_amenable}
		$\clg_{\mathrm{tight}}$ is amenable. 
	\end{ppsn}
	\begin{proof}
		We will use the criterion for amenability from Theorem \ref{crit_amenable}. The orbit space $\clg_{\mathrm{tight}}\backslash\clg_{\mathrm{tight}}^0$ is $\left\{\bbn^2, \bbn\times\{\infty\}, \{\infty\}\times\bbn,\{(\infty,\infty)\}\right\}$ (Lemma \ref{rmrkunitspace}). Observe that the quotient topology comes out to be \begin{align*}\tau=
			\Big\{
			&\emptyset,\;
			\{\mathbb{N}^2\},\;
			\{\mathbb{N}^2,\;\{\infty\}\times \mathbb{N}\},\;
			\{\mathbb{N}^2,\;\mathbb{N}\times \{\infty\}\},\;\\
			&\{\mathbb{N}^2,\;\{\infty\}\times \mathbb{N},\;\mathbb{N}\times \{\infty\}\},\;
			\{\mathbb{N}^2,\;\{\infty\}\times \mathbb{N},\;\mathbb{N}\times \{\infty\},\;\{(\infty,\infty)\}\}
			\Big\}.
		\end{align*}
		Hence $\clg_{\mathrm{tight}}\backslash\clg_{\mathrm{tight}}^0$ is $T_0$. Also, the isotropy groups $\clg_u^u\cong\bbz$ are amenable. Hence $\clg_{\mathrm{tight}}$ is amenable.
	\end{proof}
	\section{Isomorphism between $C^*(\mathcal{G}_{\mathrm{tight}})$ and $D_0$}\label{sect_isomism}
	In this section, we prove that the $C^*$-algebra $C^*(\mathcal{G}_{\mathrm{tight}})$ is isomorphic to $D_0$. This isomorphism is obtained through the induced representation $\operatorname{Ind}_{\delta_v}$ of $C^*(\mathcal{G}_{\mathrm{tight}})$ where $v=\phi(0,0)$. Using the representation $\operatorname{Ind}_{\delta_v}$, we first generate the reduced $C^*$-algebra $C^*_{\mathrm{red}}(\clg_{\mathrm{tight}})$. Since $\clg_{\mathrm{tight}}$ is amenable (Proposition \ref{g_amenable}), the reduced $C^*$-algebra $C^*_{\mathrm{red}}(\clg_{\mathrm{tight}})$ coincides with the $C^*$-algebra $C^*(\clg_{\mathrm{tight}})$.
	\begin{ppsn}
		The representation $\operatorname{Ind}_{\delta_v}:C^*(\mathcal{G}_{\mathrm{tight}})\rightarrow \mathcal{L}\left(\ell^2\left(\mathcal{G}_{\mathrm{tight},\,v}\right)\right)$ is faithful.
	\end{ppsn}
	\begin{proof}
		Let $U$ be the largest open invariant subset of $\mathcal{G}_{\mathrm{tight}}^0$ such that $\delta_v(U)=0$, which implies $(0,0)\notin U$. The orbits of $\mathcal{G}_{\mathrm{tight}}^0$ under the action of $\mathcal{G}_{\mathrm{tight}}$ are $\bbn^2$, $\mathbb{N}\times\{\infty \}$, $\{\infty\}\times\mathbb{N}$, and $\{(\infty, \infty)\}$. Since $U$ is invariant, $U$ is some union of these orbits. Because every nonempty open set in $\overline{\bbn}^2$ contains some element of $\bbn^2$, no nonempty union of these orbits is open. Hence $U=\emptyset$. Thus, by Proposition \ref{faithful}, $\operatorname{Ind}_{\delta_v}$ is faithful.
	\end{proof}
	
	From Lemma \ref{fibres 1}, we have the bijection $\Delta:\bbn\times\bbn\times\bbz\rightarrow \mathcal{G}_{\textrm{tight},\,v}$. Then $$U:\ell^2\left(\mathcal{G}_{\textrm{tight},\,v}\right)\rightarrow\ell^2(\bbn\times\bbn\times\bbz)\,;\quad U(e_z)\coloneq e_{\Delta^{-1}(z)},\quad \,\, z\in\mathcal{G}_{\textrm{tight},\,v},$$ is a unitary map for all $v\in\clg_{\mathrm{tight}}^0$. This gives an isomorphism $$\operatorname{Ad}_U:\mathcal{L}\left(\ell^2\left(\mathcal{G}_{\mathrm{tight},\,v}\right)\right)\rightarrow\mathcal{L}\left(\ell^2(\bbn\times\bbn\times\bbz)\right)\,;\quad \operatorname{Ad}_Uw\coloneq U^*w U,\quad\,\, w\in\mathcal{L}\left(\ell^2\left(\mathcal{G}_{\textrm{tight},\,v}\right)\right).$$
	Thus, we have the following diagram,
	\begin{center}
		\begin{tikzcd}[
			arrows={-{Stealth[length=9pt,width=9pt]}}, 
			line width=1.5pt,                         
			column sep=huge, row sep=huge             
			]
			C^*(\mathcal{G}_{\mathrm{tight}}) \arrow[r, "\operatorname{Ind}_{\delta_v}"] \arrow[dr, swap, "\Psi"]
			& \mathcal{L}\left(\ell^2\left(\mathcal{G}_{\textrm{tight},\,v}\right)\right) \arrow[d, "\operatorname{Ad}_U"] \\
			& \mathcal{L}(\mathbb{N} \times \mathbb{N} \times \mathbb{Z})
		\end{tikzcd}
	\end{center}
	where $\Psi :C^*(\mathcal{G}_{\mathrm{tight}})\rightarrow \mathcal{L}(\mathbb{N} \times \mathbb{N} \times \mathbb{Z})\,\,;\,\Psi \coloneq \operatorname{Ad}_U\circ \operatorname{Ind}_{\delta_v} $ is faithful.
	\begin{lmma}\label{pre lemma} 
		\begin{enumerate}[(i)]
			\item		$\phi(\bm{k})\cdot s^0_1=\phi(k_1+1,k_2+1)$ for $\bm{k}=(k_1, k_2)\in\bbn$.
			\item $\phi(k_1,0)\cdot s^0_2=\phi(k_1+1,0),\,k_1\in\bbn$ for all $k_{1}\in\mathbb{N}$.
		\end{enumerate}
	\end{lmma}	
	\begin{proof}  The proof uses case analysis to verify that the supports of both these characters are equal. We skip the details.
	\end{proof}
	\begin{thm}\label{the isomism}
		The map $\Psi$ is an isomorphism from $C^*(\mathcal{G}_{\mathrm{tight}})$ onto $D_0$.
	\end{thm}
	\begin{proof}
		From Proposition \ref{gpoid basis}, $C^*(\mathcal{G}_{\mathrm{tight}})$ is generated by $\left\{\mathbbm{1}_{\Theta_s}:s\in T\right\}$. Since $T$ is generated by the set $\left\{s_i^{0}:i=1,2,3,4\right\}$, hence, by Proposition \ref{gpoid basis} (i), $C^*(\mathcal{G}_{\mathrm{tight}})$ is generated by the set  $\left\{\mathbbm{1}_{\Theta_{s_i^{0}}}:i=1,2,3,4\right\}$. Let $\mathbbm{1}_{\Theta_{s_i^{0}}}$ be denoted by  $a_i$. We will prove that $\Psi(a_i)=s_i^{0}$ for $i=1, 2, 3,4$.
		
		Since $\Psi =\operatorname{Ad}_U\circ \operatorname{Ind}_{\delta_v}$, we start by evaluating  $\operatorname{Ind}_{\delta_v}$. From Proposition \ref{etale}, the groupoid $\clg_{\mathrm{tight}}$ is étale. In equation (\ref{discrete ind}), consider the generators $a_i\in C_c(\mathcal{G}_{\mathrm{tight}})$, the basis elements $e_{\Delta(m_0,n_0,r_0)}\in\ell^2(\mathcal{G}_{\textrm{tight},\,v})$ and the elements $\Delta(m,n,r)\in\mathcal{G}_{\textrm{tight},\,v}$ for all $v\in \clg^0_{\mathrm{tight}}$. One can verify that $\mathcal{G}^{r(\Delta(m,n,r))}_{\mathrm{tight}}=\left\{[\phi(m,n),s]\in\mathcal{G}_{\mathrm{tight}}\right\}$. Hence, equation (\ref{discrete ind}) becomes
		{\small
			\begin{equation}\label{ind compu}
				\operatorname{Ind}_{\delta_v}\left(a_i\right)e_{\Delta(m_0,n_0,r_0)}\left(\Delta(m,n,r)\right)=\sum_{\eta\in\left\{[\phi(m,n),s]\in\mathcal{G}_{\mathrm{tight}}\right\}}a_i(\eta)e_{\Delta(m_0,n_0,r_0)}\left(\eta^{-1}\Delta(m,n,r)\right).
			\end{equation}
		}
				\textbf{Claim:} $\Psi(a_{i})=s^0_i$ for $i=1,2,3,4$.
\begin{enumerate}[(i)]
\item We intend to prove that $\operatorname{Ind}_{\delta_v}\left(a_1\right)e_{\Delta(m_0,n_0,r_0)}=e_{\Delta(m_0-1,n_0-1,r_0+1)}$, which in turn implies $\Psi(a_1)=s^0_1$. From equation (\ref{ind compu}), consider the expression $a_1(\eta)$ where $\eta\in\{[\phi(m,n),s]\in\mathcal{G}_{\mathrm{tight}}\}$. From Lemma \ref{Lemma 1}, $\Theta_{s^0_1}=\left\{[\phi(k),s^0_1]\in\mathcal{G}_{\mathrm{tight}}:k\in\bbn^2\right\}$. Hence we have $a_1(\eta)=1$ if and only if $\eta=[\phi(m,n),s^0_1]$. Thus, equation (\ref{ind compu}) becomes
		\begin{align*}	
			\operatorname{Ind}_{\delta_v}\left(a_1\right)e_{\Delta(m_0,n_0,r_0)}\left(\Delta(m,n,r)\right)&=e_{\Delta(m_0,n_0,r_0)}\left([\phi(m,n),s^0_1]^{-1}\Delta(m,n,r)\right)\\&=e_{\Delta(m_0-1,n_0-1,r_0+1)}\left(\Delta(m,n,r)\right),
		\end{align*}	
	thanks to Lemma \ref{pre lemma} (we skip the trivial cases of $m_0=0$ and $n_0=0$).	Hence, we have $$\operatorname{Ind}_{\delta_v}\left(a_1\right)e_{\Delta(m_0,n_0,r_0)}=e_{\Delta(m_0-1,n_0-1,r_0+1)},\quad\Psi e_{\Delta(m_0,n_0,r_0)}=e_{\Delta(m_0-1,n_0-1,r_0+1)}.$$
		Therefore, the claim follows.
\item Similarly as above, we have $a_2(\eta)=1$ if and only if $n=0$ and $\eta=[\phi(m,0),s^0_2]$. Thus,
		\begin{equation}\label{noname2}
				\operatorname{Ind}_{\delta_v}\left(a_2\right)e_{\Delta(m_0,n_0,r_0)}\left(\Delta(m,n,r)\right)=\bbone_n e_{\Delta(m_0,n_0,r_0)}\left(\left[\phi(m,0)\cdot s^0_2,p_{0\,m+1}\otimes p\otimes t^{r-1}\right]\right).
			\end{equation}
But, from Lemma \ref{pre lemma}, we have
		\begin{align*}
			e_{\Delta(m_0,n_0,r_0)}\left(\left[\phi(m,0)\cdot s^0_2,p_{0\,m+1}\otimes p\otimes t^{r-1}\right]\right)&=e_{\Delta(m_0,n_0,r_0)}\left(\Delta(m+1,0,r-1)\right)\\
			&=e_{\Delta(m_0-1,n_0,r_0+1)}\left(\Delta(m,0,r)\right),
		\end{align*}
		and therefore, $\operatorname{Ind}_{\delta_v}(a_2)e_{\Delta(m_0,n_0,r_0)}(\Delta(m,n,r))=\bbone_n e_{\Delta(m_0-1,n_0,r_0+1)}(\Delta(m,n,r))$.
		Hence\\ $\Psi(a_2)e_{m_0n_0r_0}=\bbone_{n_0}e_{m_0-1\,0\,r_0+1}$. Thus, the claim follows.
		\item The proof of $\Psi(a_{3})=s^0_{3}$ follows similarly as in the above case.
		
		\item We have $a_4(\eta)=1$ if and only if $(m,n)=(0,0)$ and $\eta=[\phi(0,0),s^0_4]$. Also $\phi(0,0)\cdot s^0_4=\phi(0,0)$. Thus, we have
		\begin{align*}
			\operatorname{Ind}_{\delta_v}\left(a_4\right)e_{\Delta(m_0,n_0,r_0)}\left(\Delta(m,n,r)\right)&=\bbone_{(m,n)}e_{\Delta(m_0,n_0,r_0)}\left(\Delta(0,0,r-1)\right)\\
			&=\bbone_{(m,n)}e_{\Delta(m_0,n_0,r_0+1)}\left(\Delta(0,0,r)\right).\end{align*}
		Therefore, 
		$\operatorname{Ind}_{\delta_v}(a_4)e_{\Delta(m_0,n_0,r_0)}=	\bbone_{(m_0,n_0)}e_{\Delta(0,0,r_0+1)}$ and $\Psi(a_4)e_{m_0n_0r_0}= \bbone_{(m_0,n_0)}e_{0\,0\,r_0+1}$.\end{enumerate}
		Thus, the claim follows. Hence, $\Psi$ is an isomorphism from $C^*(\mathcal{G}_{\mathrm{tight}})$ onto $D_0$.
	\end{proof}
	\section{Irreducible Representations} \label{sect_rep}
	In this section, we look at the irreducible representations induced by the isotropy groups of the groupoid $C^*$-algebra $C^*(\mathcal{G}_{\mathrm{tight}})$ and show that they are equivalent to the irreducible representations of $C(SO_{q}(4)/SO_{q}(2))$ obtained in \cite{BhuBisSau-2024aaa}.\\ 
	\subsection{The Irreducible Representations of $C(SO_{q}(4)/SO_{q}(2))$}
	In this subsection, we enlist all the irreducible representations of $C(SO_{q}(4)/SO_{q}(2))$ (see \cite{BhuBisSau-2024aaa}). Consider the Weyl group $W=\left\langle \omega_1,\omega_2\,\vert\, \omega_1^2= \omega_2^2=1,\, \omega_1 \omega_2= \omega_2 \omega_1\right\rangle$. Given $t_0\in\mathbb{T}$ and $\omega\in W$,  we have the irreducible representations $\Pi_{(\omega,t_0)}$ as listed below.
	\begin{enumerate}[(1)]
		\item $\begin{aligned}[t]\Pi_{(\omega_1 \omega_2,t_0)}:&C(SO_{q}(4)/SO_{q}(2))\rightarrow\cll(\ell^2(\bbn)\otimes \ell^2(\bbn)),\\&s_1^0\mapsto t_0(S^*\otimes S^*), \,s_2^0\mapsto t_0(S^*\otimes p),\,s_3^0\mapsto t_0(p\otimes S^*),\,s_4^0\mapsto t_0(p\otimes p).\end{aligned}$
		\item $\Pi_{(\omega_1,t_0)}:C(SO_{0}(4)/SO_{0}(2))\rightarrow \cll(\ell^2(\bbn)),\quad s_1^0\mapsto t_0S^*, \,s_2^0\mapsto 0,\,s_3^0\mapsto t_0p,\,s_4^0\mapsto 0.$
		\item $\Pi_{(\omega_2,t_0)}:C(SO_{0}(4)/SO_{0}(2))\rightarrow \cll(\ell^2(\bbn)),\quad s_1^0\mapsto t_0S^*, \,s_2^0\mapsto t_0p,\,s_3^0\mapsto 0,\,s_4^0\mapsto 0.$
		\item $\Pi_{(1,t_0)}:C(SO_{0}(4)/SO_{0}(2))\rightarrow \cll(\bbc),\quad s_1^0\mapsto t_0, \,s_2^0\mapsto 0,\,s_3^0\mapsto 0,\,s_4^0\mapsto 0.$
		
	\end{enumerate}
	
	\subsection{Induced Representation}
	The purpose of this subsection is to show the equivalence of irreducible representations of $C(SO_{q}(4)/SO_{q}(2))$ and the irreducible representations induced from the isotropy groups of $\clg_{\mathrm{tight}}$.
	\begin{lmma}
		The orbits of the unit space $\clg_{\mathrm{tight}}^0$ are locally closed.
	\end{lmma}\begin{proof}Recall, from Lemma \ref{rmrkunitspace}, the orbits of $\clg_{\mathrm{tight}}^0$ are 
		\begin{align*}
			\clg_{\mathrm{tight}}^0&= \clg_1^0\sqcup\clg_2^0\sqcup\clg_3^0\sqcup\clg_4^0\\
			&\cong\{(\infty,\infty)\}\bigsqcup\bbn\times\{\infty\}\bigsqcup\{\infty\}\times\bbn\bigsqcup\bbn\times\bbn.
		\end{align*} Hence, the lemma immediately follows.\end{proof}
\begin{lmma}
	The irreducible faithful $\ast$-representations of $C^{*}(\mathbb{Z})$ are of the form $$L_{t}(f)=\sum_{n\in\mathbb{Z}}f(n)t^{n},\quad t\in\bbbt,\,f\in C_{c}(\mathbb{Z}).$$
\end{lmma}
\begin{proof}
	There is a bijection between the representations of $C^{*}(\mathbb{Z})$ and the unitary representations of $\mathbb{Z}$. Each unitary representation $U_{t}$ of $\mathbb{Z}$ is determined by a parameter $t\in\mathbb{T}$ as $$U_{t}:\mathbb{Z}\longrightarrow \mathbb{T};\,\, 1\mapsto t. $$ Thus, the representations $L_{t}$ of $C^{*}(\mathbb{Z})$ on a Hilbert space $\scrh_{L}$ over $\mathbb{C}$ are  $$L_{t}(f)=\sum_{n\in\mathbb{Z}}f(n)U_{t}(n)=\sum_{n\in\mathbb{Z}}f(n)t^{n},\quad f\in C_{c}(\mathbb{Z}).$$ \end{proof}
\noindent Given the isotropy groups $\clk_i\coloneq\clg_{\mathrm{tight},u}^u\cong\mathbb{Z}$, where $u\in \clg^{0}_i,\,i=1,2,3,4$ (see Lemma \ref{fibres 1}), the irreducible faithful $\ast$-representations of $C^{*}(\clk)$ are thus given by $L_{t}$.  For simplicity, we denote $L_{t}$ by $L$.
\begin{thm}
	The representations $\operatorname{Ind}_{\clk_i}^{\mathcal{G}}L:C^*(\clg_{\textrm{tight}})\rightarrow \mathcal{L}(\scrh_{\operatorname{Ind}L}^i)$ induced by the isotropy groups $\clk_i\coloneq\clg_{\mathrm{tight},u}^u$, for $u\in \clg^{0}_i, \,i\in\{1,2,3,4\}$, are as described below.
\begin{enumerate}[(i)] \item For $u\in\clg_1^{0}=\{\phi(\infty,\infty)\}$,
	\begin{enumerate}[a)]
		\item $\{e_r:r\in\bbz\}$ forms a basis for $\scrh_{\operatorname{Ind}L}^1$ and the inner product is $\left\langle e_{{r_{1}}},e_{{r_{2}}}\right\rangle=t_0^{r_{2}-r_{1}}$,
		\item $\scrh_{\operatorname{Ind}L}^1\cong\mathbb{C}$ with the isomorphism $e_{{r}}\mapsto t_0^{-r}$,
		\item $\operatorname{Ind}_{\clk_{1}}^{\mathcal{G}}L\cong\Pi_{(1,t_0)}$ as a representation on $\scrh^1_{\operatorname{Ind} L}$.
	\end{enumerate}
	\item For $u\in\clg_2^{0}=\{\phi(\infty,k_{2}):\,k_{2}\in\bbn\}$,
	\begin{enumerate}[a)]
		\item given $\bm{s}=(s_1,s_2)\in\bbn\times\bbz$, $\{e_{\bm{s}}:s_{2}=0\}$ forms a basis for $\scrh_{\operatorname{Ind}L}^2$, and the inner product is $\langle e_{\bm{s}},e_{\bm{u}}\rangle=\bbone_{s_{1}-u_{1}}t^{u_{2}-s_{2}}$,
		\item $\scrh^2_{\operatorname{Ind}L}\cong\mathcal{L}(\ell^2(\bbn))$ with the isomorphism $e_{\bm{s}}\mapsto t_0^{-s_{2}}e_{s_{1}}$,
		\item $\operatorname{Ind}_{\clk_{2}}^{\mathcal{G}}L\cong\Pi_{(\omega_{2},t_0)}$ as a representation on $\scrh^2_{\operatorname{Ind}L}$.
	\end{enumerate}
	\item For $u\in\clg_3^0=\{\phi(k_1,\infty):k_1\in\bbn\}$,	
	\begin{enumerate}[a)]
		\item given $\bm{s}=(s_1,s_2)\in\bbn\times\bbz$, $\{e_{\bm{s}}:s_{2}=0\}$ forms a basis for $\scrh_{\operatorname{Ind}L}^3$, and the inner product is $\langle e_{\bm{s}},e_{\bm{u}}\rangle=\bbone_{s_{1}-u_{1}}t^{u_{2}-s_{2}}$,
		\item $\scrh^3_{\operatorname{Ind}L}\cong\mathcal{L}(\ell^2(\bbn))$ with the isomorphism $e_{\bm{s}}\mapsto t_0^{-s_{2}}e_{s_{1}}$,
		\item $\operatorname{Ind}_{\clk_{3}}^{\mathcal{G}}L\cong\Pi_{(\omega_{1},t_0)}$ as a representation on $\scrh^3_{\operatorname{Ind}L}$.
	\end{enumerate}
	\item For $u\in\clg_4^{0}=\{\phi(\bm{k}):\,\bm{k}\in\bbn^2\}$,
		\begin{enumerate}[a)]
		\item given $\bm{s}\in\bbn\times\bbn\times\bbz$, $\{e_{\bm{s}}:s_{3}=0\}$ forms a basis for $\scrh_{\operatorname{Ind}L}^4$, and the inner product is 		$\left\langle e_{\bm{s}},e_{\bm{u}}\right\rangle=\bbone_{(u_{1},u_{2})-(s_{1},s_{2})}t_0^{u_{3}-s_{3}}$,
		\item $\scrh^4_{\operatorname{Ind} L}\cong \ell^{2}(\mathbb{N})\otimes\ell^{2}(\mathbb{N})$ with the isomorphism $e_{\bm{s}}\mapsto t_0^{-s_3}(e_{s_{1}}\otimes e_{s_{2}})$,
		\item $\operatorname{Ind}^\clg_{\clk_4}L\cong\Pi_{(\omega_1\omega_2,t_0)}$ as a representation on $\scrh^4_{\operatorname{Ind} L}$.
	\end{enumerate}	
	
\end{enumerate}
\end{thm}	
\begin{proof}
The representation $\operatorname{Ind}_{\clk_i}^{\clg}L$ acts on $\scrh_{\operatorname{Ind} L}^i$, which is the norm-completion of $C_{c}(\mathcal{G}_{\clk_i^{0}})\otimes\scrh_{L}$, where  $\scrh_{L}\cong\bbc$. Thus, $C_{c}(\mathcal{G}_{\clk_i^{0}})\otimes\scrh_{L}\cong C_{c}(\mathcal{G}_{\clk_i^{0}})$. Since $\clk_i=\clg_{\mathrm{tight},u}^u$, we have $\mathcal{G}_{\clk_i^{0}}=s^{-1}(\clk_i)=\mathcal{G}_{\mathrm{tight},u}$. Since $\clk_i$ is discrete, equation (\ref{pre ip}) simplifies to $$\langle \phi,\psi\rangle_{*}(k)=\sum_{y\in\mathcal{G}_{\mathrm{tight},r(k)}}\overline{\phi(y)}\psi(yk),\quad \phi,\psi\in C_{c}(\mathcal{G}_{\clk^{0}})\,,\,k\in\clk,$$ and the inner product on $\scrh_{\operatorname{Ind} L}^i$ (see equation (\ref{ip ind})) becomes
	\begin{equation}
		\langle \phi,\psi\rangle=L(\langle \phi,\psi\rangle_{*}),\quad \phi,\psi\in C_{c}(\mathcal{G}_{\mathrm{tight},u}).
	\end{equation}
	Also, in equation (\ref{ind rep}), it suffices to consider the action on generators $a_j$ of $C_{c}(\mathcal{G}_{\mathrm{tight}})$ (as in the proof of Theorem \ref{the isomism}). Hence, 
	\begin{equation}\label{ind 1}
		\operatorname{Ind}_{\clk_i}^{\mathcal{G}}L(a_j)(\phi)=a_j\ast\phi,\quad a_j\in C_{c}(\mathcal{G}_{\mathrm{tight}})\,,\,\phi\in C_{c}(\mathcal{G}_{\mathrm{tight},u}),\,j\in\{1,2,3,4\}.
	\end{equation}
\begin{enumerate}[(i)] \item Consider $u\in\clg_1^{0}=\{\phi(\infty,\infty)\}$. From Lemma \ref{fibres 1},
	$$\clk_{1}:=\clg_{\mathrm{tight},u}=\left\{\left[\phi(\infty,\infty),B_1(r, \bm{0}, (r,r)) \right] :r\in\bbz_+\right\}\cup\left\{\left[\phi(\infty,\infty),B_1(r, (-r,-r),\bm{0}) \right] :r\in\bbz_-\right\}.$$
	We can treat this set as $\bbz$ with the canonical map. Denote the elements of this set by $\sigma_r$. Let $e_{\sigma_{r_1}}, e_{\sigma_{r_2}}$ be the standard basis elements of $C_c(\clg_{\mathrm{tight},u})\cong C_c(\bbz)$. First note that, for $\sigma_{r_{3}}\in \clk_{1}$, $\clg_{\mathrm{tight},r(\sigma_{r_3})}=\{[\phi(\infty,\infty),s]\in\clg_{\mathrm{tight}}\}$. Thus,
	$$\langle e_{\sigma_{r_{1}}},e_{\sigma_{r_{2}}}\rangle_{*}(\sigma_{r_{3}})=\sum_{y\in\mathcal{G}_{\mathrm{tight},r(h)}}\overline{e_{\sigma_{r_{1}}}(y)}e_{\sigma_{r_{2}}}(y\sigma_{r_{3}}).$$
	\noindent Since the only non-trivial case is that of  $y=\sigma_{r_{1}}$, we obtain $\langle e_{\sigma_{r_{1}}},e_{\sigma_{r_{2}}}\rangle_{*}(\sigma_{r_{3}})=e_{\sigma_{r_{2}}}(\sigma_{r_{1}}\sigma_{r_{3}})$. A direct verification shows that $\sigma_{r_{1}}\sigma_{r_{3}}=\sigma_{r_{1}+r_{3}}$ for   $r_{1}\geq r_{3},|r_{1}|\geq |r_{3}|$. Therefore, $\langle e_{\sigma_{r_{1}}},e_{\sigma_{r_{2}}}\rangle_{*}(h)= \bbone_{\sigma_{r_{3}}-\sigma_{r_{2}-r_{1}}}$. This implies that $\langle e_{\sigma_{r_{1}}},e_{\sigma_{r_{2}}}\rangle_{*}=e_{\sigma_{r_{2}-r_{1}}}$. The inner product on $\scrh^1_{\operatorname{Ind} L}$ now becomes
	$$
	\left\langle e_{\sigma_{r_{1}}},e_{\sigma_{r_{2}}}\right\rangle=L(\langle e_{\sigma_{r_{1}}},e_{\sigma_{r_{2}}}\rangle_{*})
	=\sum_{n\in\mathbb{Z}}e_{\sigma_{r_{2}-r_{1}}}(\sigma_n)t_0^{n}
	=t_0^{r_{2}-r_{1}},
	$$

	\noindent since only $n=r_2-r_1$ survives. The Gram matrix of $(e_{\sigma_{r_{1}}}, e_{\sigma_{r_{2}}})$ is 
	$$\begin{pmatrix}
		1 &t_0^{r_{2}-r_{1}}\\
		t_0^{r_{1}-r_{2}} &1
	\end{pmatrix},$$
	which is not invertible. Therefore, the basis vectors $e_{\sigma_{r_{1}}}$ and $e_{\sigma_{r_{2}}}$ are linearly dependent and $\scrh_{\operatorname{Ind}L}^1$ has a single basis element. Thus, $\scrh_{\operatorname{Ind}L}^1\cong\mathbb{C}$ with the isomorphism $e_{\sigma_{r_{1}}}\mapsto t_0^{-r_{1}}$. 
	
We now explicitly verify the representation $\operatorname{Ind}_{\clk_{1}}^{\mathcal{G}}L$ from equation (\ref{ind 1}).	For the generator $a_1$, $$a_{1}\ast e_{\sigma_{r_{0}}}(\sigma_{r_{1}})=\sum_{\eta\in\mathcal{G}_{\mathrm{tight}}^{ r(\sigma_{r_{1}})}}a_{1}(\eta) e_{\sigma_{r_{0}}}(\eta^{-1}\sigma_{r_{1}}).$$ Note that $\mathcal{G}_{\mathrm{tight}}^{r(\sigma_{r_{1}})}=\{[(\phi(\infty,\infty),B_{1}(m-n,(n,n)(m,m)))]:m,n\in\mathbb{N}\}$. For $a_1(\eta)\neq0$, we get from equation (\ref{ALambdaK}), that $\eta\in\Theta_{s_1^0}=\{[\phi(j),s_1^0]:j\in\overline{\bbn}^2,\,j_1,j_2\geq1\}$. The only non-trivial contribution comes from the case  $\eta=[(\phi(\infty,\infty),s_{1}^{0})]$. Then,
	$$
		a_{1}\ast e_{\sigma_{r_{0}}}(\sigma_{r_{1}})=e_{\sigma_{r_{0}}}\left(\left[\phi(\infty,\infty),s_{1}^{0}\right]^{-1}\sigma_{r_{1}}\right)=e_{\sigma_{r_{0}}}(\sigma_{r_1+1}).
	$$	
	Thus, $a_{1}\ast e_{\sigma_{r_{0}}}=e_{\sigma_{r_{0}-1}}$ and $\operatorname{Ind}_{\clk_{1}}^{\mathcal{G}}L(a_{1})(e_{\sigma_{r_{0}}})=e_{\sigma_{r_{0}-1}}$ as a representation on $\clh^1_{\operatorname{Ind}L}$. Equivalently, as a representation on $\bbc$, $\operatorname{Ind}_{\clk_{1}}^{\mathcal{G}}L(a_{1})(t_0^{-r_0})=t_0^{-r_0+1}$. Thus, $\operatorname{Ind}_{\clk_{1}}^{\mathcal{G}}L(a_{1})=t_0$. For the generators $a_j$, $j\in\{2,3,4\}$, one can verify that $\clg_{\mathrm{tight}}^{r(\sigma_{r_1})}\cap \Theta_{s_j^0}=\emptyset$. Hence $\operatorname{Ind}^\clg_{\clk_1}L(a_j)=0$.
\item  The case $u\in\clg_2^{0}=\{\phi(\infty,k_{2}):\,k_{2}\in\bbn\}$ is analogous to the next case.
\item For $u\in\clg_3^0=\{\phi(k_1,\infty):k_1\in\bbn\}$ we have,
	{\small$$\clk_{3}=\{[\phi(k_{1},\infty),B_3(r, (k_{1},2k_{1}), (k_{1},r+2k_{1}))] :r\in \bbz_+\}\cup\{[\phi(k_{1},\infty),B_3(k_{1},2k_{1}-r),(k_{1},2k_{1})] :r\in\bbz_-\},$$} which we treat as $\mathbb{Z}$ as before, and  
	\begin{align*}
		\clg_{\mathrm{tight},u}=&\left\{\left[\phi(s_{1},\infty),B_3(s_2, (s_{1},s_{1}+k_{1}), (k_{1},s_{2}+s_{1}+k_{1}))\right] :\bm{s}=(s_{1},s_2)\in\bbn\times\bbz_+\right\}\\
		&\cup\{[\phi(s_{1},\infty),B_3(s_2,(s_{1},s_{1}+k_{1}-s_{2})),(k_{1},s_{1}+k_{1})] :\bm{s}=(s_{1},s_2)\in\bbn\times\bbz_-\},
	\end{align*}
	which we treat as $\mathbb{N}\times \mathbb{Z}$ as before. Denote the elements of $\clk_3$ and $\clg_{\mathrm{tight},u}$ as $\sigma_r$ and $\rho_{\bm{s}}$, respectively.
	The inner product on $C_{c}(\mathcal{G}_{\mathrm{tight},u})$ is given by 
	$\langle e_{\rho_{\bm{s}}},e_{\rho_{\bm{v}}}\rangle_{\ast}=e_{\rho_{\bm{v}}}(\rho_{\left(s_{1},s_{2}+r\right)}),$ using the observation $\rho_{\bm{s}}\sigma_{r}=\rho_{\left(s_{1},s_{2}+r\right)}.$
This implies that $\langle e_{\rho_{\bm{s}}},e_{\rho_{\bm{v}}}\rangle_{\ast}=\bbone_{ s_{1}-v_{1}} e_{\sigma_{v_{2}-s_{2}}}$ is the inner product on $C_{c}(\mathcal{G}_{\mathrm{tight},u})$. 

We now verify the representation $\operatorname{Ind}_{\clk_{3}}^{\mathcal{G}}L$.
	\begin{itemize}
		\item Computations similar to before yield that in the expression for $a_{1}\ast e_{\rho_{s}}(\rho_{\bm v})$, $\eta=\left[\phi(v_{1},\infty),s_{1}^{0}\right]$ for $v_{1}\geq1$, while no such $\eta$ exists for $v_{1}<1$, Since the computation for the other cases will follow similarly, we consider  $v_{2}\geq 0$. Using $v_{1}\geq 1$ and $\phi(v_{1},\infty)\cdot s_{1}^{0}=\phi(v_{1}-1,\infty)$, we get $a_{1}\ast e_{\rho_{\bm{s}}}(\rho_{\bm{v}})=e_{\rho_{s}}\left(\rho_{(v_{1}-1,v_{2}+1)}\right)$. Hence, $a_{1}\ast e_{\rho_{\bm{s}}}=e_{\rho_{\left(s_{1}+1,s_{2}-1\right)}}$. Therefore, $$\operatorname{Ind}_{\clk_{3}}^{\mathcal{G}}L(a_{1})\left(e_{\rho_{(s_{1},0)}}\right)=a_{1}\ast e_{\rho_{(s_{1},0)}}=e_{\rho_{\left(s_{1}+1,-1\right)}}.$$
		This implies $\operatorname{Ind}_{\clk_{3}}^{\mathcal{G}}L(a_{1})(e_{s_{1}})=t_{0}(e_{s_{1}+1})$. Thus, $\operatorname{Ind}_{\clk_{3}}^{\mathcal{G}}L(a_{1})=t_{0}S^{*}$.
		\item Since $\Theta_{s_{1}^{0}}\cap\mathcal{G}_{\mathrm{tight}}^{r(\rho_{v})}=\emptyset$, no such $\eta$ exists and $\operatorname{Ind}_{\clk_{3}}^{\mathcal{G}}L(a_{j})=0$ for $j\in\{2,4\}$.
		\item For $v_{1}=0, \eta=\left[\phi(0,\infty),s_{3}^{0}\right]$, and no such $\eta$ exists otherwise . Assume $v_{2}\geq 0$.  The equivalence relations from Lemma \ref{FiniteFinite} yield $a_{3}\ast e_{\rho_{\bm s}}(\rho_{v})=\bbone_{v_{1}}e_{\rho_{\bm s}}(\rho_{\left(0,v_{2}+1\right)}).$ This implies, $a_{3}\ast e_{\rho_{\bm{s}}}=\bbone_{s_{1}}e_{\rho_{\left(0,s_{2}-1\right)}}$ and hence,
		$\operatorname{Ind}_{\clk_{3}}^{\mathcal{G}}L(a_{3})\left(e_{(s_{1},0)}\right)=\bbone_{s_1}e_{\rho_{(0,-1)}}.$
		Thus, $\operatorname{Ind}_{\clk_{3}}^{\mathcal{G}}L(a_{3})=t_{0}p$.
	\end{itemize}
\item Consider $u\in\clg_4^{0}=\{\phi(\bm{k}):\,\bm{k}\in\bbn^2\}$. In this case, we have
	\begin{align*}\mathcal{G}_{\mathrm{tight},\bm{v}}&=\{[(\phi(s_1,s_2),B_{4}(s_3,(s_1,s_2),\bm{k}))]:\bm{s}=(s_1,s_2,s_3)\in \mathbb{N}\times\mathbb{N}\times\mathbb{Z}\},\\
		\clk_{4}&=\{[(\phi(k),B_{4}(r,\bm{k},\bm{k}))]:r\in\mathbb{Z}\}.\end{align*}
	Denote the elements of $\clg_{\mathrm{tight},u}$ and $\clk_4$ by $\rho_s$ and $\sigma_r$ respectively. Upon computing as before, we have $$
	\langle e_{\rho_{\bm{s}}},e_{\rho_{v}}\rangle_{*}(\sigma_{\bm{r}})=e_{\rho_{\bm{v}}}\left(\rho_{(s_{1},s_{2},r+s_{3})}\right),\quad\langle e_{\rho_{\bm{s}}},e_{\rho_{\bm{v}}}\rangle_{*}=\bbone_{(v_{1},v_{2})-(s_{1},s_{2})}e_{\sigma_{v_{3}-s_{3}}}.$$
	The result follows similarly. 
\end{enumerate}
	\end{proof}
	\textbf{Acknowledgement:} The first author gratefully acknowledges Indian Institute of Technology Gandhinagar for supporting and facilitating a major part of this manuscript. The first author also sincerely acknowledges Professor B.V.Rajarama Bhat, Indian Statistical Institute Bangalore for support through J.C.Bose fellowship during the final stages of this manuscript.

	\bigskip
	
	\bigskip
	
	\bigskip
	\bigskip

	\noindent{\sc Shreema Subhash Bhatt} (\texttt{shreemab@iitgn.ac.in},  \texttt{shreemabhatt3@gmail.com})\\
	{\footnotesize Indian Statistical Institute Bangalore, Stat. Math. Unit, R.V. College Post, Bengaluru 560059, India}\\\\
	\noindent{\sc Vinay Deshpande} (\texttt{deshpandevinay@iitgn.ac.in})\\
	{\footnotesize Department of Mathematics,  Indian Institute of Technology, Gandhinagar, Palaj, Gandhinagar 382055, India}\\\\
	\noindent{\sc Bipul Saurabh} (\texttt{bipul.saurabh@iitgn.ac.in},  \texttt{saurabhbipul2@gmail.com})\\
	{\footnotesize Department of Mathematics,  Indian Institute of Technology, Gandhinagar, Palaj, Gandhinagar 382055, India}

	\end{document}